\documentclass{article}
\usepackage{mystyle}
\usepackage[maxbibnames=99, style=numeric]{biblatex}
\bibliography{references.bib}

\usepackage[left=1in,
            right=1in,
            top=1.5in,
            bottom=1.5in]{geometry}
\usepackage{titling}

\usepackage{mathtools}
\usepackage[utf8]{inputenc}
\usepackage{todonotes}

\usepackage[hidelinks]{hyperref}
\usepackage{enumerate}
\usepackage[percent]{overpic}
\usepackage{upgreek}
\usepackage{comment}

\DeclareMathOperator{\pv}{\text{p.v.}}
\DeclareMathOperator{\calK}{\mathcal{K}}

\newcommand{\fzerone}{\mathcal{F}^{0,1}_{0}}
\newcommand{\foneone}{\mathcal{F}^{1,1}_{0}}

\newcommand{\foneonenu}{\mathcal{F}^{1,1}_{\nu}}
\newcommand{\fzeronenu}{\mathcal{F}^{0,1}_{\nu}}
\newcommand{\fsonenu}{\mathcal{F}^{s,1}_{\nu}}
\newcommand{\ftwonenu}{\mathcal{F}^{2,1}_{\nu}}
\newcommand{\Hdot}{\dot{H}^{1/2}_\nu}
\newcommand{\fonehalfonenu}{\mathcal{F}^{1/2,1}_{\nu}}
\newcommand{\fthreehalvesonenu}{\mathcal{F}^{3/2,1}_{\nu}}

\title{Global Results for the Inhomogeneous Muskat Problem}
\author{Neel Patel and Nikhil Shankar}
\date{August 1, 2021}

\begin{document}
\maketitle

\begin{abstract}

The inhomogeneous Muskat problem models the dynamics of an interface between two fluids of differing characteristics inside a non-uniform porous medium. We consider the case of a porous media with a permeability jump across a horizontal boundary away from an interface between two fluids of different viscosities and densities. For initial data of explicit medium size, depending on the characteristics of the fluids and porous media, we will prove the global existence and uniqueness of a solution which is instantly analytic and decays in time to the flat interface.
\end{abstract}

\setcounter{tocdepth}{1}

\section{Introduction}
The inhomogeneous Muskat problem models the dynamics of two incompressible, immiscible fluids in a non-uniform, porous medium. This scenario occurs naturally when, for example, oil and water flows meet in a sand and loam media. The physical principle governing the porous media flow is Darcy's law \cite{D}, given here in the two dimensional setting:
\begin{equation}\label{darcylaw}
  \frac{\mu}{\kappa} u = -\nabla p - g \begin{bmatrix} 0 \\ \rho \end{bmatrix}
\end{equation}
where $u(x,t)$ is the fluid velocity, $\mu(x,t)$ is the fluid viscosity, $\kappa(x,t)$ is the permeability of the porous media, $p(x,t)$ is the pressure, $g$ is the gravitational constant and $\rho(x,t)$ is the fluid density. The incompressibility condition in each fluid domain is given by $\nabla \cdot u = 0$.

Given non-intersecting soil and fluid interfaces, we divide our domain into three time-dependent disjoint open regions $D_{i}(t)$ such that $D_{3}(t) = D_{3}$ is unchanging.
\begin{figure}[h]
\begin{center}
    \begin{overpic}[width=0.5\textwidth]{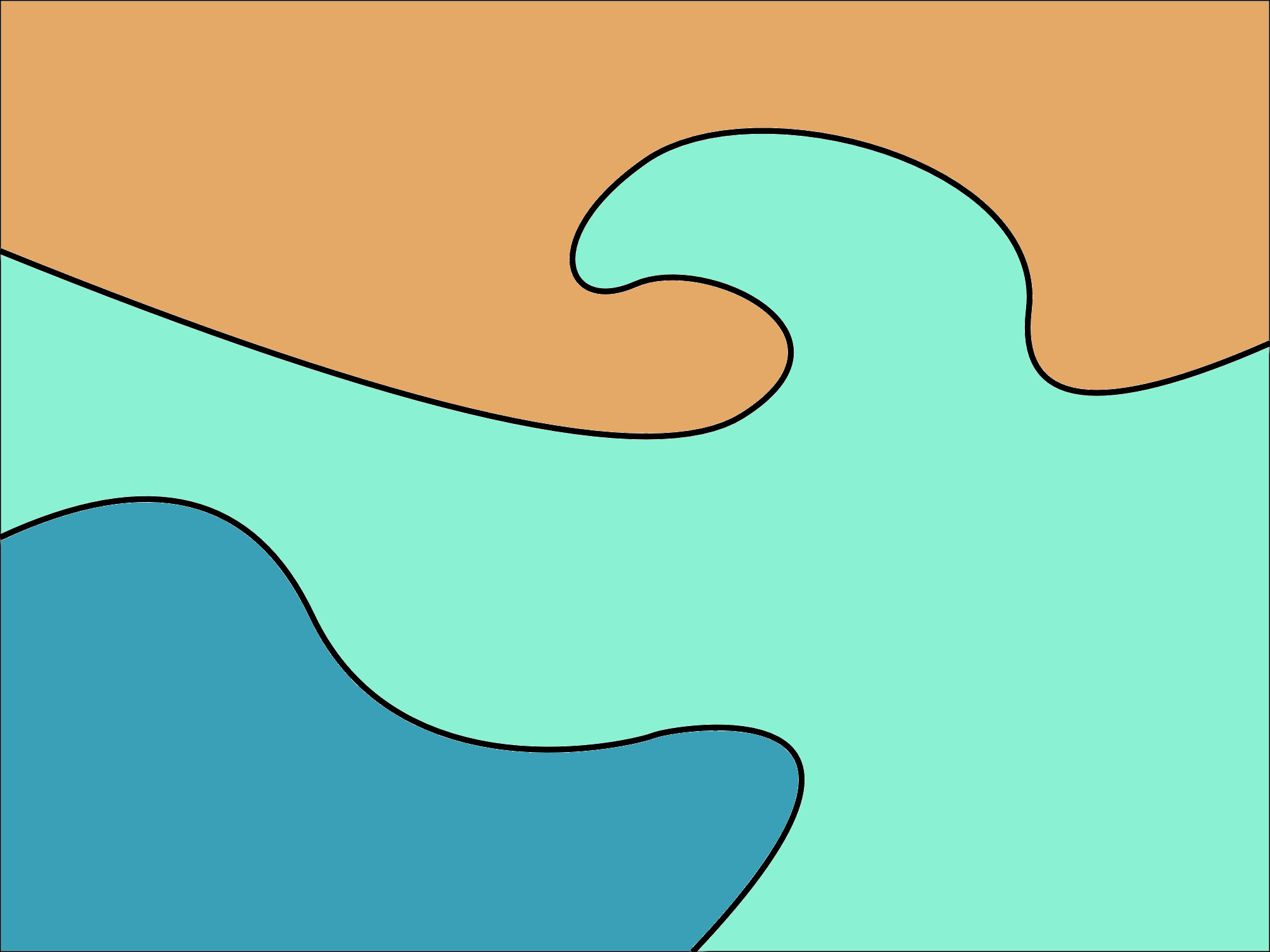}
        \put (15, 60) {\huge$\displaystyle D_1(t)$}
        \put (65, 25) {\huge$\displaystyle D_2(t)$}
        \put (10, 10)  {\huge$\displaystyle D_3$}
        \put (40, 35) {$z(\alpha, t)$}
        \put (63, 5)  {$h(\alpha)$}
    \end{overpic}
\caption{Inhomogenous Muskat Problem.}
\end{center}
\end{figure}
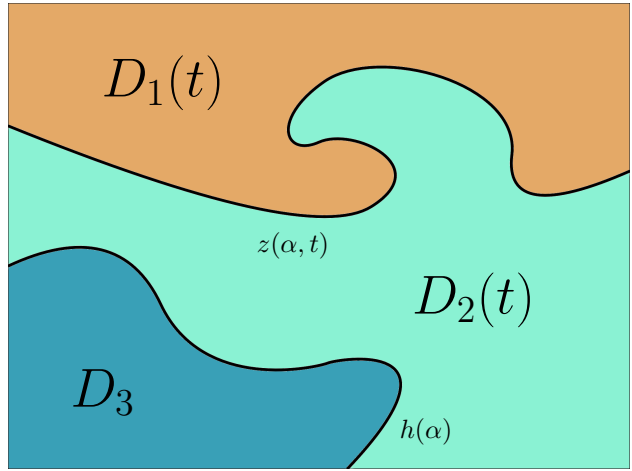

Denoting the smooth, simple boundaries as $\partial D_{i}(t)$ for $i=1,2,3$,
$$D_{1}(t)\cup D_{2}(t)\cup D_{3}\cup \partial D_{1}(t) \cup \partial D_{3} = \mathbb{R}^{2}.$$ 
The two fluids occupy domains $D_{1}(t)$ and $D_{2}(t)\cup D_{3}$ respectively, and the two different soils occupy domains $D_{1}(t)\cup D_{2}(t)$ and $D_{3}$:
\[
    (\mu, \rho)(x,t) = 
    \begin{cases}
        (\mu_1, \rho_1) & \text{if } x \in D_1(t)\\
        (\mu_2, \rho_2) & \text{if } x \in D_2(t) \cup D_3
    \end{cases}
\]
and
\[
    \kappa(x) =
    \begin{cases}
    \kappa_1 & \text{if } x \in  D_1(t) \cup D_2(t)\\
    \kappa_2 & \text{if } x \in D_3
    \end{cases}
\]
We will study the evolution of the fluid interface $\partial D_{1}(t)$. This setting is a variation of the classical Muskat problem in which the porous medium is uniform and the permeability constant is often normalized to $\kappa =1$. Note that in the case of uniform permeability, the distinction between the regions $D_{2}(t)$ and $D_{3}$ vanishes.


We obtain a contour equation for the fluid interface from \eqref{darcylaw} and the incompressibility condition on the fluid velocity by first parametrizing the fluid interface $\partial D_{1}(t)$ as
\[
    z(\alpha,t) = (z_1(\alpha,t), z_2(\alpha,t)) \qquad \text{for } \alpha\in \R,\, t \in \R_{\geq 0},
\]
and the soil interface $\partial D_{3}$ as
\[
    h(\alpha) = (h_1(\alpha), h_2(\alpha)) \qquad \text{for } \alpha \in \R
\]
as depicted in Figure 1. We assume that $z(\alpha, 0) \neq h(\alpha)$ for all $\alpha\in\mathbb{R}$.

Darcy's law implies the vorticity ($\omega = \nabla \times v$) can only be supported on the boundaries,
\[
    \omega(x, t) = \omega_1(\alpha, t) \delta(x-z(\alpha, t)) + \omega_2(\alpha, t) \delta(x-h(\alpha)).
\]
The Biot-Savart law gives a solution for the fluid velocity in terms of the vorticity
\begin{align*}
    u(x,t) &= BR(\omega_1, z)(x, t) + BR(\omega_2, h)(x,t)\\
    &\coloneqq \frac{1}{2\pi} \pv \int_\R \frac{(x-z(\beta, t))^\perp}{\abs{x-z(\beta, t)}^2}\, \omega_1(\beta, t)\; d\beta + \frac{1}{2\pi} \pv \int_\R \frac{(x - h(\beta))^\perp}{\abs{x-h(\beta)}^2}\, \omega_2(\beta, t)\; d\beta
\end{align*}
where $BR$ stands for the Birkoff-Rott integral defined above. Taking limits in the normal direction to the boundaries, one obtains
\begin{align}
    \partial_t z(\alpha, t) &= BR(\omega_1, z)(z(\alpha, t),t) + BR(\omega_2, h)(z(\alpha, t),t) + c(\alpha, t)\ \partial_\alpha z(\alpha, t), \label{zevolution}\\
    \omega_1(\alpha, t) &= 2 A_\mu (BR(\omega_1, z)(z(\alpha, t), t) + BR(\omega_2, h)(z(\alpha, t),t))\cdot \partial_\alpha z(\alpha, t) - 2A_\rho\, \partial_\alpha z_2(\alpha, t),\label{omega1z}\\
    \omega_2(\alpha, t) &= -2A_\kappa(BR(\omega_1, z)(h(\alpha), t) + BR(\omega_2, h)(h(\alpha), t))\cdot \partial_\alpha h(\alpha)\label{omega2z},
\end{align}
where the constants are given by
\begin{equation}\label{constants}
    A_\kappa = \frac{\kappa_1 - \kappa_2}{\kappa_1 + \kappa_2}, \qquad  A_\mu = \frac{\mu_1 - \mu_2}{\mu_1 + \mu_2}, \qquad A_\rho = -\kappa_1 \frac{\rho_1 - \rho_2}{\mu_1 + \mu_2}g.
\end{equation}
In \eqref{zevolution}, the tangential term $c(\alpha,t)\partial_{\alpha}z(\alpha,t)$ is determined by the choice of parametrization of the curve and vanishes when considering the velocity normal to the interface. See \cite{P-C} for more details. If there is no jump in permeability, the constant $A_{\kappa} = 0$ and the $\omega_{2}$ terms disappear from the evolution equation for $z(\alpha,t)$.

The question of well-posedness of a fluid-fluid interface in porous media has been well-studied. The local well-posedness depends on the system initially satisfying the Rayleigh-Taylor condition, which requires the jump of the gradient of the pressure in the normal direction to the interface to be strictly positive. A system satisfying the Rayleigh-Taylor condition is said to be in the stable regime, see e.g. \cite{Ambrose}. In the case that there is a density jump across the interface, $A_{\rho} \neq 0$, the Rayleigh-Taylor condition requires that the denser fluid lay below the interface, meaning $\rho_{2}>\rho_{1}$. The scaling invariance of the Muskat problem gives the criticality of $H^{1+\frac{d}{2}}$ regularity, where $d$ is the dimension of the interface. Similarly, it can be seen that $W^{1,\infty}$, $C^{1}$ and $\foneone$ are also all scale invariant. (For the definition of $\foneone$, see \eqref{Fspnu} below.)

The stable regime with uniform permeability $A_{\kappa}=0$ has been extensively studied; see for example, \cite{AbelsMatioc}, \cite{AlazardLazar}, \cite{AlazardNguyen3}, \cite{AlazardNguyen1}, \cite{AlazardNguyen2}, \cite{CCFGL}, \cite{CGS}, \cite{CGSV}, \cite{CCG2}, \cite{CCG1}, \cite{CG1}, \cite{GancedoLazar}, \cite{GGNP}, \cite{Matioc2}, \cite{matioc}, \cite{NguyenGlobal}, \cite{Nguyen}, \cite{NguyenP}, etc. and references therein. Of particular interest to this paper are results showing global well-posedness for initial data of \textit{medium size} in the case of uniform permeability. Previously, global in time results are known in both two and three dimensions for systems with uniform permeability for medium size initial data without a viscosity jump \cite{CCGRS}, \cite{CCGS}, \cite{PS} and with a viscosity jump \cite{GGPS} in the case of an infinite graph interface by using the norms \eqref{Fspnu}. Under the effects of surface tension, the stability of medium size perturbations of a gravity unstable bubble interface has also been proven in \cite{gancedo2020global}. A modulus of continuity approach in \cite{Cameron1} and \cite{Cameron2} gives a different medium size condition in both 2D and 3D without viscosity jump.

As discussed earlier, in this paper, we consider the non-uniform permeability setting, also called the inhomogeneous Muskat problem. In this setting, for a horizontal permeability jump boundary and without viscosity jump, graph interface solutions were shown to be locally well-posed using energy estimates in Sobolev spaces and there exists interfaces starting in the stable regime that become unstable in finite time \cite{BCG}. A class of graph solutions exhibiting this turning behavior was also demonstrated by a computer-assisted proof in \cite{JaviGraneroInhom}. In the same setting, \cite{inhomShkoller} demonstrated global existence and decay to the flat fluid interface for small initial data in $H^{2}$. For a general interface curve and with a viscosity jump, local well-posedness was shown in \cite{P-C} in Sobolev spaces. The existence of splash singularities was later shown in \cite{CordobaTania}, although splat singularities were ruled out. This paper will address the problem of global well-posedness in the critical regularity $\foneone$ for medium sized initial data in the regime that also has a viscosity jump. Moreover, the estimates proven in this paper will imply the interface is instantly analytic and decay to the flat solution as a corollary. The authors would like to note that while writing the final version of this paper, a concurrent result \cite{GraneroAlonso} in the regime without a viscosity jump, demonstrates precise conditions for $H^{3}$ solutions to show decay in the Lipschitz norm and further establishes a global existence and decay result for Lipschitz solutions with small initial data.

We will consider the case of a graphical fluid interface $z(\alpha,t) = (\alpha, f(\alpha,t))$ and a fixed horizontal permeability jump interface $h(\alpha) = (\alpha, -h_2)$ for $h_2 > 0$ under the assumption $\norm{f(\alpha, 0)}_{L^\infty}  < h_2$ as in \cite{BCG}, although we will allow for a viscosity jump. 
Setting
\begin{align*}
    \Delta_\beta f(\alpha) &= \frac{f(\alpha) - f(\alpha - \beta)}{\beta},
\end{align*}
the choice of a graph interface determines the tangential constant $c(\alpha,t)$ and turns the system \eqref{zevolution}-\eqref{omega1z}-\eqref{omega2z} into
\begin{align}
\label{eq: interface}
    \partial_t f(\alpha) = \frac{1}{2\pi}(I_1(\alpha) + I_2(\alpha) + I_3(\alpha) + I_4(\alpha))
\end{align}
for
\begin{align}
\label{eq: I_1}
    I_1(\alpha) &= \pv \int_\R \frac{1}{1 + \Delta_\beta f(\alpha)^2} \frac{\omega_1(\alpha - \beta)}{\beta} \; d\beta,\\
\label{eq: I_2}
    I_2(\alpha) &= \pv \int_\R \frac{ \partial_\alpha f(\alpha)\Delta_\beta f(\alpha)  }{1 + \Delta_\beta f(\alpha)^2} \frac{\omega_1(\alpha - \beta)}{\beta} \; d\beta,\\
\label{eq: I_3}
    I_3(\alpha) &=  \pv \int_\R \frac{\beta}{\beta^2 + (f(\alpha) + h_2)^2}\ \omega_2(\alpha - \beta)\; d\beta,\\
\label{eq: I_4}
    I_4(\alpha) &= \pv \int_\R \frac{\partial_\alpha f(\alpha) (f(\alpha) + h_2)}{\beta^2 + (f(\alpha) + h_2)^2}\ \omega_2(\alpha - \beta)\; d\beta,
\end{align}
in which
\begin{align}
\begin{split}
    \omega_1(\alpha) &= \frac{A_\mu}{\pi} \pv \int_\R \frac{\partial_\alpha f(\alpha) - \Delta_\beta f(\alpha) }{1 + \Delta_\beta f(\alpha)^2} \frac{\omega_1(\alpha - \beta)}{\beta}\; d\beta \\
    &\hspace{2.5em}+ \frac{A_\mu}{\pi} \pv \int_{\R} \frac{ \beta \partial_\alpha f(\alpha) -(f(\alpha) +h_2)  }{\beta^2 + (f(\alpha)+h_2)^2} \omega_2(\alpha - \beta)\; d\beta -2A_\rho \partial_\alpha f(\alpha),
    \end{split} \label{omega1equation}\\
    \omega_2(\alpha) &= -\frac{A_\kappa}{\pi} \pv \int_\R  \frac{f(\alpha - \beta) + h_2}{\beta^2 + (f(\alpha - \beta) + h_2)^2}\ \omega_1(\alpha - \beta)\; d\beta.\label{omega2equation}
\end{align}
Defining $\partial_{\alpha}\Omega_{i} = \omega_{i}$, it can be derived that
\begin{align}
    \begin{split}
        \Omega_1(\alpha) &= -\frac{A_\mu}{\pi} \int_\R \frac{\Delta_\beta f(\alpha) - \partial_\alpha f(\alpha - \beta)}{1 + \Delta_\beta f(\alpha)^2} \frac{\Omega_1(\alpha - \beta)}{\beta}\; d\beta \\
        &\hspace{2.5em}- \frac{A_\mu}{\pi} \int_\R \frac{f(\alpha) + h_2}{\beta^2 + (f(\alpha) + h_2)^2}\Omega_2(\alpha - \beta) \; d\beta - 2A_\rho f(\alpha)\label{Omega1}
    \end{split}\\
    \Omega_2(\alpha) &= -\frac{A_\kappa}{\pi} \int_\R \frac{f(\alpha - \beta) + h_2 +\beta \partial_\alpha f(\alpha - \beta)}{\beta^2 + (f(\alpha - \beta) + h_2)^2} \Omega_1(\alpha - \beta)\; d\beta. \label{Omega2}
\end{align}
The equations \eqref{eq: interface}, \eqref{omega1equation}, and \eqref{omega2equation} give a coupled system for the evolution of fluid-fluid graph interface $f(\alpha,t)$.
\section{Main Results}
To study the evolution of the interface, we adopt the weighted Fourier norms defined as follows. For a function $g:\mathbb{R}^{d} \times \R_{\geq 0} \rightarrow \mathbb{R}$, and for $s > -d$ define the norm 
\begin{equation}\label{Fspnu}
\|g\|_{\mathcal{F}^{s,p}_{\nu}} (t) = \norm{ e^{\nu t \abs{\xi}} \, |\xi|^{s}\, \hat{g}(\xi,t) }_{L^p} = \Big(\int_{\mathbb{R}^{d}} e^{\nu t  p|\xi|} |\xi|^{sp} |\hat{g}(\xi, t)|^{p} d\xi\Big)^{\frac{1}{p}}
\end{equation}
where $\hat g$ is the Fourier transform of $g$ in the spatial variable
\[
    \hat g(\xi, t) = \calF(g(\,\cdot\, , t))(\xi) = \int_{\R^d} g(x, t) e^{-ix\cdot \xi}\; dx.
\]
Let $\calF^{s,p}_{\nu}$ be the space of all functions with finite $\norm{\, \cdot\, }_{\calF^{s,p}_{\nu}}$ norm.

Let $\theta$ be the constant defined in \eqref{infconst} and observe $\theta > 0$ because $\abs{A_\kappa} < 1$ in the setting of \eqref{darcylaw}. Let $\sigma_{s} = \sigma_{s}(\|f_{0}\|_{\fzerone}, \|f_{0}\|_{\foneone})$, $s=0,1,2$ be continuous functions defined in \eqref{sigma0}, \eqref{sigma1}, \eqref{sigma2}. Note that $\sigma_{s}(0,0) = 0$. Since the $\sigma_{s}$ are continuous, we define the constants $ k_{0}(|A_{\mu}|,|A_{\kappa}|)  < h_{2}$ and $ k_{1}(|A_{\mu}|,|A_{\kappa}|)  < 1$ such that $$\theta  -  \sigma_{s}(k_{0}(|A_{\mu}|,|A_{\kappa}|),k_{1}(|A_{\mu}|,|A_{\kappa}|)) > 0$$ for $s= 0,1,2$.
\begin{theorem}\label{MainTheorem}
    Suppose $f_0 \in L^{2}\cap \foneone$ such that $\|f_{0}\|_{\fzerone} < k_{0}(|A_{\mu}|,|A_{\kappa}|)$  and $\|f_{0}\|_{\foneone} < k_{1}(|A_{\mu}|,|A_{\kappa}|)$ hold. Then there exists a unique solution $f \in L^{\infty}([0, T]; L^{2}\cap\foneone ) \cap L^{1}([0, T]; \mathcal{F}^{2,1}_{0})$ for all $T>0$ and a constant $\nu > 0$ satisfying for $s=0,1$
\begin{equation}\label{maininequality}
\norm{f}_{\mathcal{F}^{s,1}_{\nu}}(t) +(A_\rho\theta  - A_{\rho} \sigma_{s} -\nu)\int_{0}^{t} \norm{f}_{\mathcal{F}^{s+1,1}_{\nu}}(s)\; ds \leq \|f_{0}\|_{\mathcal{F}^{s,1}_{0}}
\end{equation}
and
\begin{equation}\label{L2inequality}
  \norm{f}_{L^2_\nu}^2 (t) \leq \norm{f_0}_{L^2}^2 \cdot \exp(R(\norm{f_0}_{\fzerone}, \norm{f_0}_{\foneone},t))
\end{equation}
for any $t \geq 0$ for a positive function $R$ that is bounded for $\|f_{0}\|_{\fzerone} < k_{0}(|A_{\mu}|,|A_{\kappa}|)$  and $\|f_{0}\|_{\foneone} < k_{1}(|A_{\mu}|,|A_{\kappa}|)$ and $t\geq 0$.
\end{theorem}
\begin{remark}
By \eqref{maininequality} and \eqref{L2inequality}, it can be seen from the exponential weight, that the solution gains analytic regularity instantly in time for all $t>0$.
\end{remark}
Next, to show decay to the flat solution, we will  need the Decay Lemma proved in \cite{GGPS}, which we have restated for our setting below.
\begin{lemma}[Decay Lemma]
\label{lem: decay}
    Suppose $\norm{g}_{\mathcal{F}^{s_{1},1}}(t) \leq C_0$ and
    \[
        \frac{d}{dt} \norm{g}_{\mathcal{F}^{s_{2},1}}(t) \leq -C \norm{g}_{\mathcal{F}^{s_{2}+1,1}}(t)
    \]
    where $s_1 < s_2$. Then
    \[
        \norm{g}_{\mathcal{F}^{s_{2},1}}(t) \lesssim (1+t)^{s_1 - s_2}.
    \]
\end{lemma}

This Decay Lemma along with \eqref{maininequality} implies the large time decay of solutions to the inhomogeneous Muskat problem. Specifically, \eqref{maininequality} implies uniform in time bounds of $\norm{f}_{\mathcal{F}^{s,1}}$ for $s=0,1$ and then we use the Decay Lemma to obtain the following result.
\begin{theorem}
    Suppose $f_0(\alpha)$ is initial data satisfying the conditions of Theorem \ref{MainTheorem}. Then the solution $f(\alpha,t)$ to the \eqref{eq: interface} decays with the rate
    \[
        \norm{f}_{\mathcal{F}^{1,1}}(t) \lesssim (1+t)^{-1}.
    \]
\end{theorem}

To prove Theorem \ref{MainTheorem}, we employ the following collection of useful facts. First, letting iterated convolutions be denoted as
\begin{equation*}
(\ast^{n} g)(x) = (\underbrace{g \ast \cdots *g}_{n\text{ times}})(x)
\end{equation*}
we have the product rule inequalities
\begin{lemma}
Given functions $f_{k} : \mathbb{R}^{d}\rightarrow \mathbb{R}$ for $1 \leq k \leq n$ we have
\begin{equation}\label{convolutionexp}
e^{\nu t |\xi|} (|f_{1}|\ast |f_{2}| \ast \cdots \ast |f_{n}|)  \leq (e^{\nu t |\xi|}|f_{1}|)\ast (e^{\nu t |\xi|}|f_{2}|) \ast \cdots \ast (e^{\nu t |\xi|}|f_{n}|)
\end{equation}
and for $0 < s \leq 1$
\begin{equation}\label{productrule}
|\xi|^{s} (|f_{1}|\ast |f_{2}| \ast \cdots \ast |f_{n}|)  \leq \sum_{k=1}^{n} (|\xi|^{s} |f_{k}|)\ast (\ast_{j\neq k}|f_{j}|)
\end{equation}
where
$$\ast_{j\neq k}|f_{j}|$$ indicates a convolution over the absolute values of all functions $f_{j}$ except $f_{k}$.
\end{lemma}
\begin{proof}
By the triangle inequality, we have for $0< s \leq 1$
$$ |\xi_{0}|^{s} \leq |\xi_{0}-\xi_{1}|^{s} + |\xi_{1}-\xi_{2}|^{s} + \ldots + |\xi_{n-2}-\xi_{n-1}|^{s} + |\xi_{n-1}|^{s}.$$
Applying this triangle inequality to the convolution gives us \eqref{productrule}.
Moreover,
$$e^{\nu t |\xi_{0}|} \leq \prod_{j=0}^{n} e^{\nu t |\xi_{j}-\xi_{j+1}|}.$$
Plugging this inequality into the function convolution, we obtain \eqref{convolutionexp}.
\end{proof}
We also have the interpolation inequality:
\begin{proposition}
If $g\in \mathcal{F}^{s_{1},1}_{\nu} \cap \mathcal{F}^{s_{2},1}_{\nu}$ for $s_{1} < s_{2}$, then  $g\in \mathcal{F}^{s,1}_{\nu}$ for each $s\in [s_{1},s_{2}]$ and satisfies
\begin{equation}\label{interpolation}
\|g\|_{\mathcal{F}^{s,1}_{\nu}} \leq \|g\|_{\mathcal{F}^{s_{1},1}_{\nu}}^{\theta}\|g\|_{\mathcal{F}^{s_{2},1}_{\nu}}^{1-\theta}
\end{equation}
for $\theta \in [0,1]$ such that $s = \theta s_{1} + (1-\theta) s_{2}$.
\end{proposition}
Finally, to compute the linearization of \eqref{eq: interface} and of the vorticity terms, the following Fourier transforms will be needed.

\begin{proposition}
For $a \in \mathbb{R}$, we have
\begin{equation}\label{fouriercalc1}
\calF\Big[ \frac{a}{x^{2}+a^{2}} \Big](\xi) = \pi e^{-a|\xi|}
\end{equation}
and
\begin{equation}\label{fouriercalc2}
\calF\Big[ \frac{x}{x^{2}+a^{2}} \Big](\xi) = -i\pi \sgn(\xi) e^{-a|\xi|}.
\end{equation}
\end{proposition}
\begin{proof}
    Suppose $\xi < 0$ and let $\gamma$ be the upper semi-circle contour with base $[-R, R] \subset \R \subset \C$. Using the residue theorem we can compute the contour integral
    \begin{align*}
        \int_\gamma \frac{a}{z^2 + a^2} e^{-i\xi z} \; dz = 2\pi i \cdot \text{Res}\Big(\frac{a}{z^2 + a^2} e^{-i\xi z}, ia\Big)  = 2\pi i \cdot \lim_{z \to ia} (z-ia) \frac{a}{z^2 + a^2} e^{-i\xi z} = \pi e^{-a \abs{\xi}}.
    \end{align*}
    Now, taking the limit as $R \to \infty$ the left hand side becomes 
    \begin{align*}
        \lim_{R \to \infty}
        \int_\gamma \frac{a}{z^2 + a^2} e^{-i\xi z} \; dz &= \calF\Big[ \frac{a}{x^{2}+a^{2}} \Big](\xi) + \lim_{R \to \infty} \int_{_{\substack{
        0 < \theta < \pi \\ z = Re^{i\theta}}}} \frac{a}{z^2 + a^2} e^{-i\xi z} \; dz
    \end{align*}
    where
    \[
        \Big | \int_{_{\substack{
        0 < \theta < \pi \\ z = Re^{i\theta}}}} \frac{a}{z^2 + a^2} e^{-i\xi z} \; dz \Big| \leq  \int_{_{\substack{
        0 < \theta < \pi \\ z = Re^{i\theta}}}} \Big|\frac{a}{z^2 + a^2}\Big|  e^{\xi R\sin\theta} \; Rd\theta \to 0
    \]
    since $\xi < 0$. In the case that $\xi > 0$ we instead set $\gamma$ to be the lower semi-circle contour and derive the same result. The proof of \eqref{fouriercalc2} follows this same outline.
    
\end{proof}

\subsubsection*{Outline of the Paper}
In Section \ref{potentialjuimpestimates}, we compute bounds on the vorticity terms by Taylor expanding the expressions of $\omega_{i}$ and $\Omega_{i}$ and then computing the Fourier transforms. Next, in Section \ref{sec:interfacedecay}, we decompose \eqref{eq: interface} into its linear and nonlinear parts and prove \eqref{maininequality} for $s=0,1$. In Section \ref{sec:l2est}, we prove \eqref{L2inequality} and then a higher order Sobolev estimate that is used for the existence argument. We conclude in Section \ref{sec:existence} by proving Theorem \ref{MainTheorem}.

\section{Potential Jump and Vorticity}
\label{potentialjuimpestimates}
In this section, we will decompose the potential jumps $\Omega_{1}$ and $\Omega_{2}$ and the vorticity terms $\omega_{1}$ and $\omega_{2}$ into linear and nonlinear parts. We can then use this decomposition to bound $\Omega_{i}$ and $\omega_i$ in terms of the interface function $f$.

First, we compute the Fourier transform of $\Omega_{2}$. Write $\Omega_{2} = \Omega_{21} + \Omega_{22}$ where
$$\Omega_{21}(\alpha) = -\frac{A_\kappa}{\pi} \int_\R \frac{f(\alpha - \beta) + h_2 }{\beta^2 + (f(\alpha - \beta) + h_2)^2} \Omega_1(\alpha - \beta)\; d\beta,$$
and
$$\Omega_{22}(\alpha) = -\frac{A_\kappa}{\pi} \int_\R \frac{\beta \partial_\alpha f(\alpha - \beta)}{\beta^2 + (f(\alpha - \beta) + h_2)^2} \Omega_1(\alpha - \beta)\; d\beta.$$
Taking the Fourier transform of $\Omega_{21}$, we obtain using \eqref{fouriercalc1}
\begin{align}
   \hat\Omega_{21}(\xi) &= -\frac{A_\kappa}{\pi} \int_\R \int_\R  \frac{f(\alpha - \beta) + h_2}{\beta^2 + (f(\alpha - \beta) + h_2)^2}\ \Omega_1(\alpha - \beta) e^{-i\xi \alpha}\; d\beta d\alpha \nonumber \\
    &= -\frac{A_\kappa}{\pi}\int_\R \int_\R  \frac{f(y) + h_2}{\beta^2 + (f(y) + h_2)^2}\  e^{-i\xi \beta}\; d\beta\; \Omega_1(y) e^{-i\xi y}\; dy \nonumber \\
    &= -\frac{A_\kappa}{\pi}\int_\R \pi e^{-(f(y) + h_2) \abs \xi }\cdot \Omega_1(y)  e^{-i\xi y}\; dy \nonumber \\
    &= -A_\kappa e^{-h_2 \abs \xi}  \int_\R e^{-f(y) \abs \xi} \cdot \Omega_1(y) e^{-i\xi y}\; dy \nonumber \\
    &= -A_\kappa \sum_{n=0}^{\infty}e^{-h_2 \abs \xi}  \int_\R \frac{(-f(y)|\xi|)^{n}}{n!} \cdot \Omega_1(y) e^{-i\xi y}\; dy \nonumber \\
    &= -A_\kappa\sum_{n=0}^{\infty}e^{-h_2 \abs \xi} \frac{(-1)^{n}|\xi|^{n}}{n!} \Big((\ast^{n}\hat{f}) \ast \hat \Omega_1 \Big)(\xi). \label{Omega21fourierexpansion}
\end{align}
For $\Omega_{22}$, we have
\begin{align}
    \hat{ \Omega}_{22} (\xi) &= -\frac{A_\kappa}{\pi} \int_\R d\beta \int_\R d\alpha\;  e^{-i \xi \alpha}  \frac{(\alpha - \beta) \partial_\alpha f(\beta)}{(\alpha - \beta)^2 + (f(\beta) + h_2)^2} \Omega_1(\beta)  \nonumber \\
    &= -\frac{A_\kappa}{\pi} \int_\R d\beta\;  \partial_\alpha f(\beta) \Omega_1(\beta) \calF\left[\tau_\beta\left(\frac{\alpha }{\alpha^2 + (f(\beta) + h_2)^2} \right) \right](\xi) \nonumber \\
    &= -\frac{A_\kappa}{\pi} \int_\R d\beta \; e^{-i\xi \beta} \partial_\alpha f(\beta) \Omega_1(\beta) \cdot -i\pi \sgn(\xi) e^{-(f(\beta) + h_2)\abs{\xi}} \nonumber\\
    &= A_\kappa \sum_{n= 0}^\infty i \sgn(\xi) e^{-h_2\abs{\xi} } \frac{(-1)^n \abs{\xi}^n}{n!} \Big((\ast^n \hat f) * \widehat{\partial_\alpha f} * \hat \Omega_1 \Big)(\xi). \label{Omega22fourierexpansion}
\end{align}
Next, we compute similarly for $\Omega_{1}$. By \eqref{Omega1}, we write the term $\Omega_{1}=\Omega_{11}+ \Omega_{12} - 2A_{\rho}f(\alpha)$ where
\begin{equation}\label{Omega11}
    \Omega_{11} (\alpha) = -\frac{A_\mu}{\pi} \int_\R \frac{\Delta_\beta f(\alpha) - \partial_\alpha f(\alpha - \beta)}{1 + \Delta_\beta f(\alpha)^2} \frac{\Omega_1(\alpha - \beta)}{\beta}\; d\beta,
\end{equation} 
and
$$\Omega_{12}(\alpha) = - \frac{A_\mu}{\pi} \int_\R \frac{f(\alpha) + h_2}{\beta^2 + (f(\alpha) + h_2)^2}\Omega_2(\alpha - \beta) \; d\beta.$$
It can be seen that $\Omega_{11}$ has no part that is linear in $\Omega_{i}$. For $\Omega_{12}$, taking the Fourier transform, we obtain
\begin{align}
    \hat{\Omega}_{12}(\xi)
    &= -\frac{A_{\mu}}{\pi}\int_\R e^{-i\xi \alpha} \int_\R \int_\R \hat \Omega_2(\xi_1) e^{i\xi_{1}(\alpha-\beta)}d\xi_{1}\frac{f(\alpha) + h_2}{\beta^2 + (f(\alpha) + h_2)^2}\; d\beta d\alpha \nonumber\\
    &= -\frac{A_{\mu}}{\pi}\int_\R e^{-i\xi \alpha} \int_\R  \hat \Omega_2(\xi_1) e^{i\xi_{1}\alpha} \left(\int_\R e^{-i\xi_{1}\beta}\frac{f(\alpha) + h_2}{\beta^2 + (f(\alpha) + h_2)^2}\; d\beta \right) d\xi_{1} d\alpha \nonumber\\
     &= -\frac{A_{\mu}}{\pi}\int_\R e^{-i\xi\alpha} \int_\R e^{i \xi_1 \alpha} \; \hat\Omega_2(\xi_1) \ \pi e^{-(f(\alpha) + h_2)\abs{\xi_1}}\; d\xi_1 d\alpha \nonumber\\
    &= -A_{\mu}\sum_{n = 0}^{\infty} \frac{(-1)^n}{n!} \int_\R  e^{-i\xi\alpha} f(\alpha)^n \int_\R \hat \Omega_2(\xi_1) \  e^{-h_2\abs{\xi_1} } \ \abs{\xi_1}^n  e^{i \xi_1 \alpha}  \;d\xi_1 d\alpha  \nonumber\\
    &=-A_{\mu} \sum_{n = 0}^{\infty} \frac{(-1)^n}{n!}\Big((\hat \Omega_2 \  e^{-h_2\abs{\xi } }\ \abs{\xi}^n ) * (\ast^n \hat f) \Big)(\xi). \label{omega12fourierlinearpartsum}
\end{align}
The first term $\Omega_{11}$ does not have an explicit computation as the other terms, but it satisfies the following bound. The proof technique is used from \cite{CCGS}.
\begin{lemma}
For $\norm{f}_{\foneone} < 1$, we have the bound
\begin{equation}\label{Omega11abs}
    |\hat \Omega_{11}(\xi)| \leq  2 |A_\mu| \sum_{n=0}^{\infty}(\ast^{2n+1} |\widehat{\partial_\alpha f}| \ast \abs{\hat \Omega_1})(\xi).
\end{equation}
\end{lemma}
\begin{proof}
We first consider the term:
\begin{align*} \Omega_{11}(\alpha) &= -\frac{A_\mu}{\pi} \int_\R \frac{\Delta_\beta f(\alpha) - \partial_\alpha f(\alpha - \beta)}{1 + \Delta_\beta f(\alpha)^2} \frac{\Omega_1(\alpha - \beta)}{\beta}\; d\beta\\
&\eqdef \Omega_{111}(\alpha) + \Omega_{112} (\alpha).
\end{align*}
Taking the Fourier transform of the first term $\Omega_{111}$ and Taylor expanding the denominator for $|\Delta_\beta f(\alpha)| < 1$, we obtain
\begin{align*}
   \hat \Omega_{111} (\xi) &= -\frac{A_{\mu}}{\pi}\int_\R \calF \left[ \frac{\Delta_\beta f(\alpha)}{1 + \Delta_\beta f(\alpha)^2} \frac{\Omega_1(\alpha - \beta)}{\beta} \right](\xi)\; d\beta\\
    &= -\frac{A_{\mu}}{\pi}\sum_{n = 0}^\infty (-1)^n \int_\R \calF\left[ \Delta_\beta f(\alpha)^{2n+1} \frac{\Omega_1(\alpha - \beta)}{\beta}\right](\xi) \; d\beta \\
    &= -\frac{A_{\mu}}{\pi}\sum_{n = 0}^\infty  (-1)^n \int_\R (\ast^{2n+1} (\hat f m_\beta) * \widehat{\tau_\beta \Omega_1})(\xi)\; \frac{d\beta}{\beta}
\end{align*}
in which
\[
    m_\beta(\xi) = \frac{1 - e^{-i\xi\beta}}{\beta}.
\]
Expanding the convolution of the $n$-th term in the sum, we have
\begin{align*}
    \int_\R (\ast^{2n+1} (\hat f m_\beta) * \widehat{\tau_\beta \Omega_1})(\xi) \frac{d\beta}{\beta} = \int_\R d\xi_1 \ldots \int_\R d\xi_{2n+1} \hat f(\xi - \xi_1) \cdot \ldots \cdot \hat f(\xi_{2n} - \xi_{2n+1}) \cdot \hat \Omega_1(\xi_{2n+1}) \cdot M_n
\end{align*}
where 
\begin{align*}
    M_n = M_n(\xi, \xi_1, \ldots, \xi_{2n+1}) = \int_\R d\beta \frac{e^{-i\xi_{2n+1} \beta}}{\beta} m_\beta(\xi - \xi_1) \cdot \ldots \cdot m_{\beta}(\xi_{2n} - \xi_{2n+1} ).
\end{align*}
Since
\[
    m_\beta(\xi) = i\xi \int_0^1 ds\; e^{i\beta(s-1)\xi}
\]
we have
\begin{align*}
    |M_n| &= \Big|i^{2n}(\xi - \xi_1) \cdot \ldots \cdot (\xi_{2n} - \xi_{2n+1}) \int_0^1 ds_1 \ldots \int_0^1 ds_{2n+1} \int_\R d\beta \frac{e^{-i\xi_{2n+1} \beta}}{\beta} e^{i \beta A}\Big|\\
    &= \Big| (\xi - \xi_1) \cdot \ldots \cdot (\xi_{2n} - \xi_{2n+1}) \int_0^1 ds_1 \ldots \int_0^1 ds_{2n+1}\  i\pi \sgn(A- \xi_{2n+1})\Big|\\
    &\leq \pi \cdot \abs{\xi - \xi_1} \cdot \ldots \cdot \abs{\xi_{2n} - \xi_{2n+1}}
\end{align*}
where 
\[
    A = (s_1 - 1) (\xi - \xi_1) + \ldots + (s_{2n+1} - 1) (\xi_{2n} - \xi_{2n+1}).
\]
Therefore,
\begin{align*}
    \dabs{\int_\R (\ast^{2n+1} (\hat f m_\beta) * \widehat{\tau_\beta \omega_1})(\xi) \frac{d\beta}{\beta}}
    &\leq \int_\R d\xi_1 \ldots \int_\R d\xi_{2n+1} \abs{ \hat f(\xi - \xi_1)} \cdot \ldots \cdot \abs{\hat f(\xi_{2n} - \xi_{2n+1})} \cdot \abs{\hat \Omega_1(\xi_{2n+1})} \cdot \abs{M_n} \\
    &\leq \pi (\ast^{2n+1} |\widehat{\partial_\alpha f}| \ast \abs{\hat \Omega_1})(\xi).
\end{align*}
The term $\Omega_{112}$ is bounded by the same quantity using a similar computation. This concludes the proof.
\end{proof}

Next, we consider the vorticity terms. The Fourier transform of $\omega_{2}$ can be computed similarly to $\Omega_{21}$

\begin{equation}\label{omega2fourierexpansion}
   \hat\omega_2(\xi) = -A_\kappa\sum_{n=0}^{\infty}e^{-h_2 \abs \xi} \frac{(-1)^{n}|\xi|^{n}}{n!} \Big((\ast^{n}\hat{f}) \ast \hat \omega_1 \Big)(\xi). 
\end{equation}
For $ \omega_{1}$, similarly to $\Omega_{1}$, we decompose $\omega_{1} = \omega_{11} + \omega_{12} -2A_{\rho}\partial_{\alpha}f$ where we have the analogous bound on $\hat \omega_{11}$
\begin{equation}\label{omega11est}
|\hat\omega_{11}(\xi)| \leq 2 |A_\mu|  \sum_{n=0}^{\infty}(\ast^{2n+1} |\widehat{\partial_\alpha f}| \ast \abs{\hat \omega_1})(\xi)
\end{equation}
and, $\hat \omega_{12}(\xi)$ is explicitly computed as $-i\xi \hat \Omega_{12}(\xi)$ using \eqref{productrule}
\begin{align}
\hat \omega_{12}(\xi) &=A_{\mu}\sum_{n = 0}^\infty \frac{(-1)^{n+1}}{n!} \Big((\ i \sgn(\cdot) e^{-h_2 \abs{\cdot }} \ \abs{ \cdot }^n \hat{\omega}_2) \ast \widehat{\partial_\alpha f} * (\ast^n \hat f) \Big)(\xi) \label{omega12calc} \\ &\hspace{1in} - A_{\mu}\sum_{n = 0}^{\infty} \frac{(-1)^n}{n!}\Big((\hat \omega_2 \ e^{-h_2\abs{\cdot } }\ \abs{\cdot}^n ) * (\ast^n \hat f) \Big)(\xi) \nonumber
\end{align}
Using the computations above, we can bound the vorticity terms in the frequency space norms $\fsonenu$ for $s=0,1$.
\begin{proposition}
The term $\omega_2$ satisfies

\begin{align}\label{lem:omega2fzeroneinitial}
     \|\omega_{2}\|_{\fzeronenu} &\leq   \abs{A_\kappa}C_{0} \|\omega_1\|_{\fzeronenu}
\end{align}
and
\begin{align}\label{lem:omega2foneoneinitial}
     \|\omega_{2}\|_{\foneonenu} &\leq   \abs{A_\kappa} \frac{C_{2}}{\norm{f}_{\fzeronenu}} \|\omega_1\|_{\fzeronenu}\|f\|_{\foneonenu} + \abs{A_\kappa}C_{0} \|\omega_1\|_{\foneonenu}
\end{align}
where 
\begin{equation}\label{C0}C_{0} = C_{0}(\|f\|_{\fzeronenu}) = \sum_{n = 0}^\infty \dnorm{ e^{-h_2\abs\xi} \frac{\abs{\xi}^n}{n!}}_{L^\infty} \norm{f}_{\fzeronenu}^n = \sum_{n = 0}^\infty \frac{n^{n}}{e^nn!} \left(\frac{\norm{f}_{\fzeronenu}}{h_2}\right)^n
\end{equation}
and 
\begin{equation}\label{C2}
C_2 = \sum_{n=1}^{\infty} n \dnorm{ e^{-h_2\abs\xi} \frac{\abs{\xi}^n}{n!}}_{L^\infty} \norm{f}_{\fzeronenu}^n = \sum_{n = 1}^\infty  \frac{n^{n+1}}{e^nn!} \left(\frac{\norm{f}_{\fzeronenu}}{h_2}\right)^n
\end{equation}
which converge for $\norm{f}_{\fzeronenu} < h_{2}$. Note that $C_0 \to 1$ and $C_2 \to 0$ in the limit $\norm{f}_{\fzeronenu} \to 0$.
\end{proposition}
\begin{proof}
Using Young's inequalities for convolutions, \eqref{convolutionexp}, and \eqref{omega2fourierexpansion}:
\begin{align*}
\norm{\omega_2}_{\fzeronenu} &\leq \abs{A_\kappa}   \sum_{n=0}^{\infty} \Big\| e^{\nu t \abs{\xi}} e^{-h_2 \abs \xi} \frac{(-1)^{n}|\xi|^{n}}{n!} \Big((\ast^{n}\hat{f}) \ast \hat \omega_1 \Big)(\xi)\Big\|_{L^1}\\
&\leq  \abs{A_\kappa} \sum_{n=0}^{\infty} \dnorm{ e^{-h_2 \abs \xi} \frac{\abs{\xi}^{n}}{n!}}_{L^\infty} \Big\| e^{\nu t \abs{\xi}} (\ast^{n}\hat{f}) \ast \hat \omega_1 \Big\|_{L^1}\\
&\leq \abs{A_\kappa}\Big( \sum_{n = 0}^\infty \dnorm{ e^{-h_2 \abs \xi} \frac{\abs{\xi}^{n}}{n!}}_{L^\infty} \|f\|_{\fzeronenu}^n \Big) \|\omega_1\|_{\fzeronenu}.
\end{align*}
This computation yields \eqref{lem:omega2fzeroneinitial}. Applying \eqref{productrule}, we obtain \eqref{lem:omega2foneoneinitial}.
\end{proof}
We will use \eqref{lem:omega2fzeroneinitial} and \eqref{lem:omega2foneoneinitial} implicitly to bound $\omega_{1}$ in the same norms.
\begin{lemma}
The vorticity term $\omega_{1}$ satisfies the bounds
\begin{equation}\label{omega1fzeronenubound}
\|\omega_1\|_{\fzeronenu} \leq 2A_\rho C_1\norm{f}_{\foneonenu}
\end{equation}
and
\begin{equation}\label{omega1foneonenubound}
\|\omega_1\|_{\foneonenu} \leq 2A_\rho C_1 C_3  \norm{f}_{\ftwonenu}
\end{equation}
where 
\begin{equation}\label{C1}
     C_1 = \left(1- \abs{A_\mu} \left[\frac{2\norm{f}_{\foneonenu}}{1 - \norm{f}_{\foneonenu}^2} +   \abs{A_\kappa} C_0^2  (1 + \norm{f}_{\foneonenu})\right]\right)\inv
     \end{equation}
     and
\begin{equation}\label{C3}
C_3 =  1 + 2 \abs{A_\mu} C_1 \bigg(\frac{\norm{f}_{\foneonenu} (1 + \norm{f}_{\foneonenu}^2)}{(1 - \norm{f}_{\foneonenu}^2)^2} + \frac{1}{2} \abs{A_\kappa} C_0 \Big[(C_0 + 2C_2) \norm{f}_{\foneonenu} + C_2(1 +  \norm{f}_{\foneonenu}  )\Big] \bigg)
\end{equation}
 are defined for $\|f\|_{\fzeronenu}< k_0(|A_{\mu}|,|A_{\kappa}|)$ and $\|f\|_{\foneonenu} <  k_1(|A_{\mu}|,|A_{\kappa}|)$. As $\norm{f}_{\fzeronenu} + \norm{f}_{\foneonenu} \to 0$, we have $C_1 \to (1- \abs{A_\kappa} \abs{A_\mu} )\inv $ and $C_3 \to 1$. 
\end{lemma}
\begin{proof}
Similarly to the proof of \eqref{lem:omega2fzeroneinitial}, we use \eqref{convolutionexp} and Young's inequality to obtain from \eqref{omega11est} that
\begin{align}
\|\omega_{11}\|_{\fzeronenu} &\leq 2|A_{\mu}| \sum_{n=0}^{\infty}\|f\|_{\foneonenu}^{2n+1}\|\omega_{1}\|_{\fzeronenu}\nonumber\\
&\leq |A_{\mu}|\frac{2\norm{f}_{\foneonenu}}{1 - \norm{f}_{\foneonenu}^2}    \|\omega_{1}\|_{\fzeronenu}.\label{omega11fzerone}
\end{align}
We can also bound $\omega_{12}$ from \eqref{omega12calc} by 
\begin{align}
\|\omega_{12}\|_{\fzeronenu} &\leq |A_{\mu}|C_{0}(1+\|f\|_{\foneonenu})\|\omega_{2}\|_{\fzeronenu} \nonumber\\
&\leq |A_\kappa| |A_{\mu}|C_{0}^{2}(1+\|f\|_{\foneonenu})\|\omega_1\|_{\fzeronenu} \label{omega12fzerone}
\end{align}
where we used \eqref{lem:omega2fzeroneinitial} in the second inequality. Hence, we now have that
$$\|\omega_{1}\|_{\fzeronenu} \leq 2 |A_{\mu}|\frac{\norm{f}_{\foneonenu}}{1 - \norm{f}_{\foneonenu}^2} \|\omega_{1}\|_{\fzeronenu} + |A_\kappa |  |A_{\mu}|C_{0}^{2}(1+\|f\|_{\foneonenu})\|\omega_{1}\|_{\fzeronenu} + 2A_\rho\|f\|_{\foneonenu}.$$
Using \eqref{lem:omega2fzeroneinitial} solving for $\|\omega_{1}\|_{\fzeronenu}$ in the inequality implies \eqref{omega1fzeronenubound}. The estimate \eqref{omega1foneonenubound} follows similarly by applying \eqref{productrule}. Next, via \eqref{interpolation} we compute 
\begin{align*}
\norm{\omega_{11}}_{\foneonenu} &\leq 2|A_{\mu}|\frac{1+\norm{f}_{\foneonenu}^{2}}{(1 - \norm{f}_{\foneonenu}^2)^{2}}\|f\|_{\ftwonenu}    \|\omega_{1}\|_{\fzeronenu} + |A_{\mu}|\frac{2\norm{f}_{\foneonenu}}{1 - \norm{f}_{\foneonenu}^2}    \|\omega_{1}\|_{\foneonenu},\\
\norm{\omega_{12}}_{\foneonenu} &\leq \abs{A_\mu} (C_0 + C_2)\norm{\omega_2}_{\fzeronenu}\norm{f}_{\ftwonenu} + \abs{A_\mu} C_0   \norm{f}_{\foneonenu}\norm{\omega_2}_{\foneonenu}\\
&\hspace{5em} + \abs{A_\mu} \Big( \frac{C_2}{\norm{f}_{\fzeronenu}} \norm{f}_{\foneonenu} \norm{\omega_2}_{\fzeronenu}  + C_0 \norm{\omega_2}_{\foneonenu} \Big).
\end{align*}
Now, using \eqref{omega1fzeronenubound} and \eqref{lem:omega2fzeroneinitial} gives \eqref{eq:omega2fzerone} from the next proposition. Using \eqref{eq:omega2fzerone}, \eqref{lem:omega2foneoneinitial}, \eqref{omega1fzeronenubound} and applying the interpolation \eqref{interpolation}, we obtain
\begin{align}
\norm{\omega_{11}}_{\foneonenu} &\leq 4 A_\rho \abs{A_\mu} C_1 \frac{\norm{f}_{\foneonenu} (1 + \norm{f}_{\foneonenu}^2)}{(1- \norm{f}_{\foneonenu}^2)^2} \norm{f}_{\ftwonenu} + 2\abs{A_\mu} \frac{\norm{f}_{\foneonenu}}{1 - \norm{f}_{\foneonenu}^2} \norm{\omega_1}_{\foneonenu} \label{omega11foneone}\\
\norm{\omega_{12}}_{\foneonenu} &\leq 2 A_\rho \abs{A_\kappa} \abs{A_\mu} C_0 C_1 (C_0 + 2C_2) \norm{f}_{\foneonenu}\norm{f}_{\ftwonenu} +  \abs{A_\kappa} \abs{A_\mu} C_0^2(1 + \norm{f}_{\foneonenu}) \norm{\omega_1}_{\foneonenu}\nonumber\\
&\hspace{5em} + 2 A_\rho \abs{A_\kappa}\abs{A_\mu} C_0 C_1 C_2 (1 + \norm{f}_{\foneonenu}) \norm{f}_{\ftwonenu}  \label{omega12foneone}.
\end{align}
Computing implicitly as before yields the bound.
\end{proof}

Plugging the estimates \eqref{omega1fzeronenubound} and \eqref{omega1foneonenubound} into \eqref{lem:omega2fzeroneinitial} and \eqref{lem:omega2foneoneinitial}, and then using \eqref{interpolation} we obtain
\begin{proposition}
The term $\omega_{2}$ is bounded as
   \begin{align}
\|\omega_{2}\|_{\fzeronenu} &\leq   2  A_\rho \abs{A_\kappa} C_0 C_1 \|f\|_{\foneonenu}  \label{eq:omega2fzerone}\\
\|\omega_2\|_{\foneonenu} &\leq 2 A_\rho \abs{A_\kappa} C_1 C_4 \norm{f}_{\ftwonenu}\label{omega2foneone}
\end{align} 
where
\begin{equation}\label{C4}
C_{4} = C_{2} + C_{0}C_{3} \to 1 \quad \text{as} \quad \norm{f}_{\fzeronenu} + \norm{f}_{\foneonenu} \to 0.
\end{equation}
\end{proposition}

\section{Instant Analyticity and Decay Inequality for the Interface}
\label{sec:interfacedecay}
In this section, we will prove the inequality \eqref{maininequality} which will imply the instantaneous gain of analytic regularity and the decay to the flat solution of the density jump interface for initial data of an explicitly calculable size. We perform these estimates in the spaces defined in \eqref{Fspnu} for $p=1$, $\nu > 0$ and $0\leq s \leq 1$. We first need to linearize the contour equation for the fluid-fluid interface $f(\alpha,t)$. To do so, we need to extract the linear part of $\omega_{i}$. First, by \eqref{omega2fourierexpansion}, we can write the decomposition in of $\omega_{2}$ in frequency space:
\begin{equation}\label{omega2linearnonlinear}
\hat{\omega}_{2}(\xi) = L(\hat\omega_{2})(\xi) + N(\hat\omega_{2})(\xi)
\end{equation}
where
$$L(\hat\omega_{2})(\xi) = -A_\kappa e^{-h_2 \abs \xi}  \hat \omega_{1}(\xi)$$
and
\begin{align*}
N(\hat\omega_{2})(\xi) &= -A_\kappa\sum_{n=1}^{\infty}e^{-h_2 \abs \xi} \frac{(-1)^{n}|\xi|^{n}}{n!} \Big((\ast^{n}\hat{f}) \ast \hat \omega_1 \Big)(\xi)
\end{align*}
From \eqref{omega12calc}, we can write
\begin{equation}\label{Omega1linearnonlinear}
\hat \omega_{1}(\xi) = L(\hat\omega_{1})(\xi) + N(\hat\omega_{1})(\xi)
\end{equation}
where
$$L(\hat \omega_{1})(\xi) = -A_\mu e^{-h_2 \abs \xi}  \hat \omega_{2}(\xi) - 2A_{\rho}\widehat{\partial_{\alpha}f}(\xi)$$
and
\begin{align}
\label{NOmega1hat}
N(\hat \omega_{1})(\xi) &= \hat \omega_{11}(\xi) -A_{\mu} \sum_{n = 1}^{\infty} \frac{(-1)^n}{n!}\Big((\hat \omega_2 \  e^{-h_2\abs{\xi } }\ \abs{\xi}^n ) * (\ast^n \hat f) \Big)(\xi)\\
&\hspace{1in}+ A_{\mu}\sum_{n = 0}^\infty \frac{(-1)^{n+1}}{n!} \Big((\ i \sgn(\xi) e^{-h_2 \abs{\xi }} \ \abs{ \xi }^n \hat \omega_2) \ast\widehat{\partial_\alpha f} * (\ast^n \hat f) \Big)(\xi) \nonumber.
\end{align}
We can now use the relations \eqref{omega2linearnonlinear} and \eqref{Omega1linearnonlinear} to say
\begin{align*}
    L(\hat \omega_2) &= -A_\kappa e^{-h_2 \abs{\xi} } ( L(\hat \omega_1)(\xi) + N(\hat \omega_1)(\xi) )\\
    &= A_\kappa A_\mu e^{-2h_2\abs{\xi} }(L(\hat \omega_2)(\xi)+ N(\hat \omega_2)(\xi)) + 2A_\kappa A_\rho e^{-h_2 \abs{\xi} } \widehat{\partial_{\alpha}f}(\xi) - A_\kappa e^{-h_2 \abs{\xi}} N(\hat \omega_1)
\end{align*}
Combining terms and solving for $L(\hat \omega_2)$ we find
\begin{equation}\label{Omega2linearpart}
    L(\hat \omega_2)(\xi) = 2A_\kappa A_\rho  \widehat{\partial_{\alpha}f}(\xi) \frac{e^{h_2 \abs{\xi}}}{e^{2h_2\abs{\xi}} - A_\kappa A_\mu} - A_\kappa N(\hat \omega_1)(\xi) \frac{e^{h_2 \abs{\xi}}}{e^{2h_2\abs{\xi}} - A_\kappa A_\mu} - N(\hat \omega_2)(\xi) \frac{A_\kappa A_\mu}{e^{2h_2\abs{\xi}} - A_\kappa A_\mu}.
\end{equation}
With \eqref{Omega2linearpart} in hand, we are ready to begin analyzing the interface decay. Differentiating in time, we obtain
\begin{align*}
    \frac{d}{dt} \norm{f}_{\fsonenu} &= \frac{d}{dt} \int_\R e^{\nu t\abs{\xi} } \abs{\xi}^s|\hat{f}(\xi)|\; d\xi \\
    &= \nu \int_\R e^{\nu t\abs{\xi} } \abs{\xi}^{s+1}|\hat{f}(\xi)|\; d\xi + \frac{1}{2}\int_\R e^{\nu t\abs{\xi} } \abs{\xi}^s\frac{\hat{f}(\xi)\overline{\widehat{\partial_{t}f}(\xi)} 
    +\widehat{\partial_{t}f}(\xi)\overline{\hat{f}(\xi)}}{|\hat{f}(\xi)|}\; d\xi.
\end{align*}
Next, to obtain the decay term in the expression above, we need to decompose the evolution equation for the interface into the linear and nonlinear terms
\begin{align*}
    \widehat{\partial_t f}(\xi) &= \frac{1}{2\pi} ( \hat I_1(\xi) + \hat I_2(\xi) + \hat I_3(\xi) + \hat I_4(\xi ) ).
\end{align*}
Defining $N_0$ as
\begin{align*}
    I_1(\alpha) &= \pv \int_\R \frac{\omega_1(\alpha - \beta)}{ \beta}\; d\beta - \int_\R \frac{\Delta_\beta f(\alpha)^2}{1 + \Delta_\beta f(\alpha)^2} \frac{\omega_1(\alpha - \beta)}{\beta}\; d\beta\\
    &\eqdef \pi H(\omega_1)(\alpha) + 2\pi N_0(\alpha),
\end{align*}
where $H$ is the Hilbert transform, we use equation \eqref{Omega1linearnonlinear} to find
\begin{align}
    \begin{split}
    \label{eq:Ionehat_first}
        \frac{1}{2\pi} \hat I_1(\xi) &= \frac{1}{2} \widehat{H(\omega_1)}(\xi) + \hat N_0(\xi) \\
        &= -A_\rho |\xi| \hat f(\xi) + \frac{i}{2}\sgn(\xi) A_\mu e^{-h_2|\xi|} \hat \omega_2(\xi) -  \frac{i}{2}\sgn(\xi) N( \hat \omega_1)(\xi) +  \hat N_0(\xi).
    \end{split}
\end{align}
Expanding $\hat \omega_2$ via \eqref{omega2linearnonlinear} and \eqref{Omega2linearpart} gives
\begin{align}
    \label{Ionehat_second}
    \frac{i}{2}\sgn(\xi) A_\mu e^{-h_2|\xi|} \hat \omega_2(\xi)  
    &= -A_\rho |\xi| \hat f(\xi) \frac{A_\kappa A_\mu}{e^{2h_2\abs{\xi}} - A_\kappa A_\mu} - A_\mu(\hat N_1 + \hat N_2 + \hat N_3)(\xi)
\end{align}
in which
\begin{align}
\label{N1}
    \hat N_1(\xi) &= -\frac{i}{2}\sgn(\xi) e^{-h_2\abs{\xi}} N(\hat \omega_2)(\xi),\\
\label{N2}
    \hat N_2(\xi) &= \frac{i}{2} \sgn(\xi) \frac{A_\kappa}{e^{2h_2\abs{\xi}} - A_\kappa A_\mu} N(\hat \omega_1)(\xi), \\
\label{N3}
    \hat N_3(\xi) &= \frac{i}{2} \sgn(\xi) \frac{A_\kappa A_\mu e^{-h_2\abs{\xi}}}{e^{2h_2\abs{\xi}} - A_\kappa A_\mu}N(\hat \omega_2)(\xi).
\end{align}
Combining \eqref{eq:Ionehat_first} and \eqref{Ionehat_second} leads to
\begin{align*}
    \frac{1}{2\pi}\hat I_1(\xi) = -A_\rho |\xi| \hat f(\xi) \left (1 + \frac{A_\kappa A_\mu}{e^{2h_2 \abs{\xi}} - A_\kappa A_\mu} \right) +  (\hat N_0 + \hat N_1 + \hat N_2 + \hat N_3)(\xi) - \frac{i}{2} \sgn(\xi) N(\hat\omega_1)(\xi) .
\end{align*}

Next,
\begin{align}
\label{I3linearnonlinear}
     I_3(\alpha) &= \int_\R \frac{\beta}{\beta^2 + h_2^2}\, \omega_2(\alpha - \beta) \; d\beta - \int_\R \frac{f(\alpha) + 2h_2}{\beta^2 + (f(\alpha) + h_2)^2} \frac{\beta f(\alpha) }{\beta^2 + h_2^2}\, \omega_2(\alpha - \beta)\; d\beta,
\end{align}
denoting the nonlinear part as 
\[
     N_4 = -\frac{1}{2\pi} \int_\R \frac{f(\alpha) + 2h_2}{\beta^2 + (f(\alpha) + h_2)^2} \frac{\beta f(\alpha) }{\beta^2 + h_2^2}\, \omega_2(\alpha - \beta)\; d\beta,
\]
we apply \eqref{fouriercalc2} and \eqref{Omega2linearpart} once again
\begin{align}
    \frac{1}{2\pi} \hat I_3(\xi) &= -\frac{i}{2} \sgn(\xi) e^{-h_2\abs\xi} \hat \omega_2(\xi) + \hat N_4(\xi) \nonumber\\
     \label{eq:I3hat}
    &= A_\rho |\xi| \hat f(\xi) \frac{A_\kappa }{e^{2h_2\abs{\xi}} - A_\kappa A_\mu} +  (\hat N_1 + \hat N_2 + \hat N_3 + \hat N_4)(\xi).
\end{align}
Collecting terms we have
\begin{align}
\label{interface_neat}
\begin{split}
    \widehat{\partial_t f}(\xi) &= - A_\rho |\xi| \hat f(\xi) \left(1 - \frac{A_\kappa(1-A_\mu) }{e^{2h_2\abs{\xi}} - A_\kappa A_\mu}\right) + \frac{1}{2\pi}(\hat I_2 + \hat I_4)(\xi)\\
    & \hspace{2.5em} +\hat N_0 (\xi) +  (1 + A_\mu)(\hat N_1 + \hat N_2 + \hat N_3)(\xi)  + \hat N_4(\xi) - \frac{i}{2} \sgn(\xi) N(\omega_1)(\xi).
    %
\end{split}
\end{align}

In \eqref{interface_neat}, the linear terms will give the decay of the interface as long as the nonlinear terms are sufficiently bounded. So let us bound the nonlinear terms by following analogous computations to those in Section \ref{potentialjuimpestimates}. In all the nonlinear bounds in this section, the constants grow arbitrarily small as $\norm{f}_{\fzeronenu} + \norm{f}_{\foneonenu} \to 0$.

Defining
\begin{align}
\label{C5}
C_5 &= \bnorm{\frac{A_\kappa A_\mu e^{-h_2\abs{\xi}}}{e^{2h_2\abs{\xi}} - A_\kappa A_\mu}}_{L^\infty}\\
\label{C6}
C_{6} &= \frac{\|f\|_{\foneonenu}}{1-\|f\|_{\foneonenu}^{2}} \to 0 \quad \text{as} \quad \norm{f}_{\foneonenu} \to 0
\end{align}
we have the estimates
\begin{align}\label{I2fzeroone}
    \frac{1}{2\pi}\norm{I_2}_{\fzeronenu} &\leq \frac{1}{2}  \frac{\norm{f}_{\foneonenu}^2} {1- \norm{f}_{\foneonenu}^2} \norm{\omega_1}_{\fzeronenu}   \leq
     A_\rho C_1 C_6 \norm{f}_{\foneonenu}^2, \\
     \frac{1}{2\pi}\norm{I_2}_{\foneonenu} &\leq \frac{\norm{f}_{\foneonenu}}{( 1 - \norm{f}_{\foneonenu}^2)^2}( \norm{\omega_1}_{\fzeronenu} \norm{f}_{\ftwonenu} + \frac{1}{2}(1- \norm{f}_{\foneonenu}^2) \norm{f}_{\foneonenu} \norm{\omega_1}_{\foneonenu}) \nonumber\\
     &\leq   2A_\rho C_1(1 + \frac{1}{2}C_3(1 - \norm{f}_{\foneonenu}^2))C_6^2 \norm{f}_{\ftwonenu}.    \label{I2foneone}
\end{align}
Next, similar to \eqref{omega12calc},
\begin{align*}
    \hat I_4(\xi) = \pi \sum_{n = 0}^{\infty} \frac{(-1)^n}{n!}\Big((\hat \omega_2 \ e^{-h_2\abs{\cdot } }\ \abs{\cdot}^n ) * (\ast^n \hat f)  * \widehat{\partial_\alpha f} \Big)(\xi)
\end{align*}
the bounds on $I_{4}$ are
\begin{align}
    \frac{1}{2\pi }\norm{I_4}_{\fzeronenu} & \leq \frac{1}{2} C_0 \norm{f}_{\foneonenu} \norm{\omega_2}_{\fzeronenu} \nonumber \\
    &\leq   A_\rho \abs{A_\kappa} C_0^2 C_1 \norm{f}_{\foneonenu}^2    \label{I4fzeroone}\\
    \frac{1}{2\pi}\norm{I_4}_{\foneonenu} & \leq A_\rho \abs{A_\kappa} C_0 C_1 (C_0 + C_2) \norm{f}_{\foneonenu}\norm{f}_{\ftwonenu} + C_0 \norm{f}_{\foneonenu} \norm{\omega_2}_{\foneonenu} \nonumber \\
    &\leq  A_\rho \abs{A_\kappa} C_0C_1(C_0 + C_2 +C_4) \norm{f}_{\foneonenu}\norm{f}_{\ftwonenu} \nonumber\\
    &\eqdef A_\rho \abs{A_\kappa} C_1\lambda_0   \norm{f}_{\ftwonenu}  \label{I4foneone}
\end{align}
where
\begin{equation}\label{lambda0}
\lambda_{0} =  C_0(C_0 + C_2 +C_4)\norm{f}_{\foneone} \to 0 \quad \text{as} \quad \norm{f}_{\fzeronenu} + \norm{f}_{\foneonenu} \to 0.
\end{equation}
Next, like the term $\Omega_{111}$, 
\[
    |\hat N_0| \leq \frac{1}{2} \sum_{n =0}^\infty ((\ast^{2n+2} |\widehat{\partial_\alpha f}|)*|\hat \omega_1|)(\xi) 
\]
which leads to
\begin{align*}
    \norm{N_0}_{\fzeronenu} &\leq A_\rho C_1 C_6 \norm{f}_{\foneonenu}^2\\
    \norm{N_0}_{\foneonenu} &\leq A_\rho C_1 C_{6}\Big(C_3 \norm{f}_{\foneonenu} + \frac{2C_6}{\norm{f}_{\foneonenu}} \Big)\norm{f}_{\ftwonenu}.
\end{align*}
Now, reusing techniques from the Fourier transforms in section \ref{potentialjuimpestimates}, we have
\begin{align*}
    -2\pi \hat N_{4}(\xi) &= \calF \left[\int_\R d\beta \; \frac{f(\alpha) + 2h_2}{\beta^2 + (f(\alpha) + h_2)^2} \frac{\beta f(\alpha) }{\beta^2 + h_2^2} \omega_2(\alpha - \beta) \right](\xi)\\
    &= \int_\R d\beta \;  \left( \calF\left(\frac{f(\alpha) + 2h_2}{\beta^2 + (f + h_2)^2} \frac{\beta f(\alpha)}{\beta^2 + h_2^2} \right) * \widehat{\tau_\beta \omega_2}\right)(\xi) \\
    &= \int_\R d\beta \int_\R d\xi_1  \; \calF\left( \frac{(f(\alpha)  + 2h_2)}{\beta^2 + (f(\alpha) + h_2)^2} \frac{\beta f(\alpha)}{\beta^2 + h_2^2} \right)(\xi - \xi_1) \   \ e^{-i\beta \xi_1}\hat \omega_2(\xi_1)\\
    &= \int_\R d\xi_1 \;  \hat \omega_2 (\xi_1) \int_\R d\alpha \; e^{-i(\xi - \xi_1)\alpha}f(\alpha)  \int_\R d\beta \;e^{-i\beta \xi_1} \frac{f(\alpha) + 2h_2}{\beta^2 + (f(\alpha) + h_2)^2} \frac{\beta }{\beta^2 + h_2^2}.
\end{align*}
By \eqref{fouriercalc1} and \eqref{fouriercalc2} the integral in $\beta$ is the convolution
\begin{align*}
   T(\xi_1) &\eqdef \int_\R d\beta \;e^{-i\beta \xi_1} \frac{f(\alpha) + 2h_{2}}{\beta^2 + (f(\alpha) + h_2)^2} \frac{\beta}{\beta^2 + h_2^2}\\
   &= \Bigg(\frac{\pi (f(\alpha)+2h_2)}{f(\alpha) + h_2} e^{-(f(\alpha) + h_2)\abs{\cdot} } \ast -i\pi \sgn(\cdot) e^{-h_2 \abs{\cdot}}\Bigg)(\xi_1)
\end{align*}
which can be calculated via the identity
\begin{align*}
    \Big( e^{-a |\cdot| } * \sgn(\cdot) e^{-b |\cdot| } \Big) (x)= - \sgn(x) 
    \frac{2a(e^{-a|x|} - e^{-b|x|})}{a^2-b^2}
\end{align*}
as 
\begin{align*}
    T(\xi_1) &= i\pi^2 \frac{ f(\alpha) + 2h_2}{f(\alpha) + h_2}\cdot \sgn(\xi_1) \frac{2(f(\alpha) + h_2)(e^{-(f(\alpha) + h_2)\abs{\xi_1}} - e^{-h_2|\xi_1|})}{(f(\alpha) + h_2)^2 - h_2^2}\\
    &= 2i\pi^2 \sgn(\xi_1) e^{-h_2|\xi_1|}  \frac{(e^{-f(\alpha)\abs{\xi_1}} - 1)}{f(\alpha)}\\
    &= 2i\pi^2 \sgn(\xi_1) e^{-h_2|\xi_1|} \sum_{n=1}^\infty \frac{(-f(\alpha))^{n-1} |\xi_1|^n}{n!} 
\end{align*}
and so
\begin{align*}
     \hat N_{4}(\xi) &= -i\pi\sum_{n=1}^\infty  \frac{(-1)^{n-1}}{n!} \int_\R d\xi_1 \; \hat \omega_2(\xi_1) \sgn(\xi_1) \abs{\xi_1}^n e^{-h_2|\xi_1|} \int_\R d\alpha \; e^{-i(\xi - \xi_1)\alpha}f(\alpha)^{n}.
\end{align*}
Hence,
\begin{align}
\label{N4hat}
    \hat N_4(\xi) = i\pi \sum_{n=1}^\infty \frac{(-1)^n}{n!} \Big( (\hat \omega_2\ \sgn(\xi) \ \abs{\xi}^n\ e^{-h_2 \abs{\xi}} ) * (\ast^n \hat f) \Big) (\xi).
\end{align}
Using \eqref{interpolation} gives us the bounds
\begin{align*}
    \norm{N_4}_{\fzeronenu} &\leq \pi (C_0 - 1) 
    \norm{\omega_2}_{\fzeronenu} \nonumber\\
    &\leq 2\pi A_\rho \abs{A_\kappa} C_0(C_0-1)C_1 \norm{f}_{\foneonenu},\\
    \norm{N_4}_{\foneonenu} &\leq 2\pi A_\rho \abs{A_\kappa} C_0 C_1 C_2 \norm{f}_{\ftwonenu} + \pi (C_0 - 1) \norm{\omega_2}_{\foneonenu} \nonumber\\
    &\leq 2\pi A_\rho \abs{A_\kappa} C_1(C_0 C_2 + (C_0 - 1)C_4)\norm{f}_{\ftwonenu}.
\end{align*}
At this point, bounds on $N(\omega_1)$ and $N(\omega_2)$ will lead to quick estimates of all the remaining terms. In $\calF^{0,1}$
\begin{align*}
    \norm{N(\omega_2)}_{\fzeronenu} &\leq \abs{A_\kappa} (C_0 - 1) \norm{\omega_1}_{\fzeronenu} \nonumber\\
    &\leq 2A_\rho | A_\kappa| (C_0 - 1) C_1\norm{f}_{\foneonenu}
\end{align*}
and
\begin{align*}
    \norm{N(\omega_1)}_{\fzeronenu} 
    &\leq 2 |A_\mu| C_6 \norm{\omega_1}_{\fzeronenu} + |A_\mu| (C_0 \norm{f}_{\foneonenu} + (C_0-1) ) \norm{\omega_2}_{\fzeronenu} \nonumber\\
    &\leq A_\rho |A_\mu|C_1 \lambda_1  \norm{f}_{\foneonenu}
\end{align*}
where the second inequality follows from  \eqref{omega1fzeronenubound} and \eqref{eq:omega2fzerone}. Above,
\begin{align*}
    \lambda_1 = 4C_6 + 2\abs{A_\kappa} C_0 (C_0 \norm{f}_{\foneonenu} + (C_0 - 1)) \to 0 \quad \text{as} \quad \norm{f}_{\fzeronenu} + \norm{f}_{\foneonenu} \to 0.
\end{align*}
Next in $\calF^{1,1}$
\begin{align*}
    \norm{N(\omega_2)}_{\foneonenu} &\leq \abs{A_\kappa} \Big((C_0 - 1) \norm{\omega_1}_{\foneonenu} + \frac{C_2}{\norm{f}_{\fzeronenu}} \norm{f}_{\foneonenu} \norm{\omega_1}_{\fzeronenu} \Big)\\
    &\leq A_\rho \abs{A_\kappa}  C_1 \lambda_2 \norm{f}_{\ftwonenu}
\end{align*}
where $\lambda_2$ is defined by applying \eqref{omega1fzeronenubound}, \eqref{omega1foneonenubound}, and then \eqref{interpolation}
\begin{align*}
    \lambda_2 = 2(C_0 - 1)C_3 + 2C_2 \to 0 \quad \text{as} \quad \norm{f}_{\fzeronenu} + \norm{f}_{\foneonenu} \to 0.
\end{align*}
Finally, calculating as in \eqref{omega12foneone},
\begin{align*}
\begin{split}
      \norm{N(\omega_1)}_{\foneonenu}  &\leq  \norm{\omega_{11}}_{\foneonenu} + |A_\mu| \Big( (C_0 + C_2) \norm{\omega_2}_{\fzeronenu} \norm{f}_{\ftwonenu} + C_0 \norm{f}_{\foneonenu} \norm{\omega_2}_{\foneonenu} +\\
      & \hspace{12em} \frac{C_2}{\norm{f}_{\fzeronenu} } \norm{f}_{\foneonenu} \norm{\omega_2}_{\fzeronenu} + (C_0 - 1)\norm{\omega_2}_{\foneonenu} \Big)
\end{split}\\
      &\leq A_\rho \abs{A_\mu} C_1 \lambda_3 \norm{f}_{\ftwonenu}
\end{align*}
in which $\lambda_3$ is defined by using \eqref{omega1fzeronenubound}, \eqref{omega1foneonenubound}, \eqref{omega11foneone}, \eqref{eq:omega2fzerone}, \eqref{omega2foneone},
\begin{align*}
\lambda_3 &= 4 \frac{1 + \norm{f}_{\foneonenu}^2} {(1-\norm{f}_{\foneonenu}^2)^2} \norm{f}_{\foneonenu}  + 4C_3 C_6 + 2\abs{A_\kappa} (C_0(C_0+C_2 + C_4) \norm{f}_{\foneonenu}\\
&\hspace{5em} + 2\abs{A_\kappa}( C_0C_2 + (C_0-1) C_4),
\end{align*}
and $\lambda_3 \to 0$ as $\norm{f}_{\fzeronenu} + \norm{f}_{\foneonenu} \to 0$. 

Now, fix
\begin{align}
\label{infconst}
    \theta = \inf_{\xi \in \R} \left(1 - \frac{A_\kappa(1-A_\mu) }{e^{2h_2\abs{\xi}} - A_\kappa A_\mu}\right) 
\end{align}
and note that $\theta > 0$ since $|A_\kappa| < 1$. Then by \eqref{interface_neat} and the above estimates
\begin{align*}
    \frac{d}{dt} \norm{f}_{\fzeronenu} &\leq (-A_\rho\theta + \nu) \norm{f}_{\foneonenu} + A_\rho \sigma_0 \norm{f}_{\foneonenu}
\end{align*}
where
\begin{align}
\label{sigma0}
    \sigma_0 &\eqdef \sigma_{0}(\|f\|_{\fzeronenu}, \|f\|_{\foneonenu}) \nonumber\\ 
    \begin{split}
    &= C_1 \Big[ (C_6 + \abs{A_\kappa} C_0^2) \norm{f}_{\foneonenu}\\
    &\hspace{2.5em} + 
    \frac{1}{2} \abs{1+ A_\mu} \left( 2|A_\kappa| (C_0 - 1) + C_5 \Big(2 \abs{A_\kappa} (C_0 -1) + \lambda_1 \Big) \right)\\
    &\hspace{3.5em} + 2\pi \abs{A_\kappa} C_0(C_0-1) + \abs{A_\mu}\lambda_1 +C_6 \norm{f}_{\foneonenu} \Big]
    \end{split}
\end{align}
and
\begin{align*}
    \frac{d}{dt} \norm{f}_{\foneonenu} &\leq (-A_\rho\theta + \nu) \norm{f}_{\ftwonenu} + A_\rho \sigma_1 \norm{f}_{\ftwonenu}
\end{align*}
where
\begin{align}
\label{sigma1}
    \sigma_1&\eqdef \sigma_{1}(\|f\|_{\fzeronenu}, \|f\|_{\foneonenu}) \nonumber\\ 
    \begin{split}
    &= C_1 \Big[ 2(1+\frac{1}{2}C_3 (1-\norm{f}_{\foneonenu}^2))C_6^2 + |A_\kappa| \lambda_0 \\
    &\hspace{2.5em} + \frac{1}{2} \abs{1+ A_\mu} \left( |A_\kappa| \lambda_2 + C_5 \Big(|A_\kappa| \lambda_2 + \lambda_3 \Big) \right)\\
    &\hspace{3.5em} + 2\pi \abs{A_\kappa} (C_0 C_2 + (C^2_0 - 1)C_4) + \frac{1}{2} \abs{A_\mu} \lambda_3 + C_3C_6 \norm{f}_{\foneonenu} + \frac{2C_6^2}{\norm{f}_{\foneonenu}}\Big].
    \end{split}
\end{align}
Note that $\sigma_{i}(\|f\|_{\fzeronenu}, \|f\|_{\foneonenu})$, $i=1,2$, are continuous functions in $(\|f\|_{\fzeronenu}, \|f\|_{\foneonenu})$ such that $\sigma_{i}(0,0) = 0$.

\section{$L^{2}$ Estimates}
\label{sec:l2est}

\subsection{Analytic Estimates}
In this section, we will prove the $L^{2}$ estimate \eqref{L2inequality} of Theorem \ref{MainTheorem}. We will introduce the notation $\lesssim$ which indicates a constant depending on $\norm{f}_{\fzeronenu}$ and $\norm{f}_{\foneonenu}$. Let us begin by differentiating
\begin{align}
\label{interfaceL2}
    \frac{1}{2} \frac{d}{dt}\norm{f}_{L^2_\nu}^2(t) = \nu \norm{f}_{\Hdot}^2 - A_\rho \int_\R e^{2\nu t |\xi|} \abs{\xi} \abs{\hat f(\xi)}^2 \left(1- \frac{A_\kappa(1-A_\mu)}{e^{2h_2\abs{\xi}} - A_\kappa A_\mu}\right) \; d\xi + \langle \hat f, \text{h.o.t.} \rangle_{L^2_\nu}
\end{align}
in which the higher order terms are 
\begin{align*}
\begin{split}
    \text{h.o.t.} &= \frac{1}{2\pi}(\hat I_2 + \hat I_4)
     + \hat N_0(\xi) + (1 + A_\mu)(\hat N_1 + \hat N_2 + \hat N_3)(\xi)  + \hat N_4(\xi) - \frac{i}{2} \sgn(\xi) N(\hat \omega_1)(\xi).
\end{split}
\end{align*}
In the following estimates, we will use the convention $\tilde{g}(x) = g(-x)$ and the convolution identities
\begin{align*}
   \int_{\mathbb{R}} g_{1}(x)(g_{2}\ast g_{3})(x) dx &= \int_{\mathbb{R}} g_{1}(x)\int_{\mathbb{R}}g_{2}(x-y)g_{3}(y) dy dx\\
   &= \int_{\mathbb{R}}g_{3}(y)  \int_{\mathbb{R}}g_{1}(x) g_{2}(x-y)dx dy\\
   &=  \int_{\mathbb{R}} g_{3}(y)(g_{1}\ast \tilde{g_{2}})(y) dy
\end{align*}
and
\begin{align*}
\widetilde{g_{1}\ast g_{2}} = \tilde{g_{1}}\ast \tilde{g_{2}}.
\end{align*}
Using the substitution $\abs{\hat \omega_1(\xi)} = \abs{\xi} \abs{\hat \Omega_1(\xi)}$, we obtain
\begin{align}
    \langle \hat f, \hat I_2 \rangle_{L^2_\nu} &\leq \int_\R e^{2\nu t\abs{\xi} } \abs{\hat f(\xi)} \cdot \pi \sum_{n=0}^\infty \Big((\ast^{2n+2} \abs{\widehat {\partial_\alpha f}}) * \abs{\cdot} \abs{\hat \Omega_1} \Big)(\xi)\; d\xi \nonumber\\
    &\leq \pi \sum_{n=0}^\infty \int_\R e^{\nu t\abs{\cdot} } \abs{\hat f(\xi)} \Big((\ast^{2n+2} e^{\nu t\abs{\cdot} }\abs{\widehat {\widetilde{\partial_\alpha f}}}) * e^{\nu t\abs{\cdot} }\abs{\cdot} \abs{\hat \Omega_1} \Big)(\xi)\; d\xi \nonumber\\
    &\leq \pi \sum_{n=0}^\infty \int_\R e^{\nu t\abs{\cdot} } \abs{\xi}\abs{\hat \Omega_1(\xi)} \Big((\ast^{2n+2} e^{\nu t\abs{\cdot} }   \abs{\widehat {\widetilde{\partial_\alpha f}}}) * e^{\nu t\abs{\cdot} }\abs{\hat f} \Big)(\xi)\; d\xi \nonumber\\
    &\leq \pi \sum_{n=0}^\infty \norm{\Omega_1}_{\Hdot} \bnorm{\abs{\xi}^{1/2}\Big((\ast^{2n+2} e^{\nu t\abs{\cdot} }   \abs{\widehat {\widetilde{\partial_\alpha f}}}) * e^{\nu t\abs{\cdot} }\abs{\hat f} \Big)}_{L^2} \nonumber\\
    &\leq \pi  \sum_{n=0}^\infty  \norm{\Omega_1}_{\Hdot}  \norm{f}_{\foneonenu}^{2n+1}\Big((2n+2) \norm{f}_{\calF^{3/2,1}_\nu} \norm{f}_{L^2_\nu} + \norm{f}_{\foneonenu} \norm{f}_{\Hdot} \Big)\nonumber\\
    \begin{split}
    \label{youngsexample}
    &\leq \pi \sum_{n=0}^\infty \frac{\epsilon_n}{2} \norm{\Omega_1}_{\Hdot}^2 + \frac{1}{2\epsilon_n} (2n+2)^2 \norm{f}_{\foneonenu}^{4n+2} \norm{f}_{\calF^{3/2,1}_\nu}^2 \norm{f}_{L^2_\nu}^2\\
    &\hspace{5em} + \norm{f}_{\foneonenu}^{2n+2} \norm{\Omega_1}_{\Hdot}\norm{f}_{\Hdot}
    \end{split}
\end{align}
where the last inequality is from Young's inequality for products where we choose $\epsilon_n = \epsilon/(1+n^2)$ for a small value $\epsilon> 0$. 
The other non-linear terms can be estimated via similar methods
\begin{align*}
    \begin{split}
    \langle \hat f, \hat I_4 \rangle_{L^2_\nu} &\leq \pi \sum_{n=0}^\infty \dnorm{ e^{-h_2 \abs \xi} \frac{\abs{\xi}^{n}}{n!}}_{L^\infty} \norm{\Omega_2}_{\Hdot} \norm{f}_{\fzeronenu}^{n}\Big( \norm{f}_{\fthreehalvesonenu} \norm{f}_{L^2_\nu}
    + (n+1) \norm{f}_{\foneonenu}
    \norm{f}_{\Hdot}\Big)
    \end{split}
\end{align*}
and
\begin{align*}
    \langle \hat f, \hat N_1 \rangle_{L^2_\nu} &\leq \frac{1}{2} \int_{\R} d\xi \; e^{2\nu t\abs{\xi}} e^{-h_2\abs{\xi}} \abs{\hat f(\xi)} \abs{N(\widehat{\partial_\alpha \Omega_2})(\xi)} \\
    &\leq \frac{\abs{A_\kappa}}{2} \sum_{n=1}^{\infty} \int_{\R} d\xi \; e^{2\nu t\abs{\xi}} e^{-2h_2\abs{\xi}} \frac{|\xi|^{n}}{n!} \cdot \abs{\hat f(\xi)}   \Big((\ast^{n}\abs{\hat{f}}) \ast \abs{\cdot} \abs{\hat \Omega_1} \Big)(\xi) \\
    &\leq  \frac{\abs{A_\kappa}}{2} \sum_{n=1}^{\infty} \dnorm{e^{-2h_2\abs{\xi}} \frac{|\xi|^{n}}{n!}}_{L^\infty} \int_{\R} d\xi \; e^{2\nu t\abs{\xi}} \abs{\xi} \abs{\hat \Omega_1(\xi)}   (\ast^{n}\abs{\hat{\tilde{f}}} * \abs{\hat f})(\xi) \\
    &\leq \frac{\abs{A_\kappa}}{2} \sum_{n=1}^{\infty} \dnorm{e^{-2h_2\abs{\xi}} \frac{|\xi|^{n}}{n!}}_{L^\infty} (n+1)\norm{f}_{\fzeronenu}^n \norm{\Omega_1}_{\Hdot} \norm{f}_{\Hdot}\\
    &\eqdef \abs{A_\kappa} \lambda_4 \norm{\Omega_1}_{\Hdot} \norm{f}_{\Hdot}
\end{align*}
where $\lambda_4 \to 0$ as $\norm{f}_{\fzeronenu} + \norm{f}_{\foneonenu} \to 0$. Next
\begin{align}
\begin{split}
    \label{prodfN2}
    \langle \hat f, \hat N_2 \rangle_{L^2_\nu} &\leq \frac{1}{2} C_5 \Big( \norm{\Omega_1}_{\Hdot}\sum_{n=0}^\infty 2 (\norm{f}_{\foneonenu}^{2n+1} \norm{f}_{\Hdot} + (2n+1) \norm{f}_{\foneonenu}^{2n} \norm{f}_{\fthreehalvesonenu}\norm{f}_{L^2_\nu}) \\
    &\hspace{3em} + \norm{\Omega_1}_{\Hdot}\sum_{n= 0}^\infty \dnorm{ e^{-h_2|\xi|} \frac{|\xi|^n}{n!} }_{L^\infty} \norm{f}_{\fzeronenu}^n( \norm{f}_{\fthreehalvesonenu} \norm{f}_{L^2_\nu} + (n+1) \norm{f}_{\foneonenu} \norm{f}_{\Hdot}) \\
    &\hspace{5em} + \norm{\Omega_2}_{\Hdot}\sum_{n=1}^\infty  \dnorm{ e^{-h_2|\xi|} \frac{|\xi|^n}{n!} }_{L^\infty}   (n+1) \norm{f}_{\fzeronenu}^n \norm{f}_{\Hdot} \Big)
    \end{split}\\
    &\eqdef \frac{1}{2} C_5\left(\norm{\Omega_1}_{\Hdot} (\lambda_5 \norm{f}_{\Hdot} + \lambda_6 \norm{f}_{\fthreehalvesonenu}\norm{f}_{L^2_\nu} ) + \lambda_7  \norm{\Omega_2}_{\Hdot} \norm{f}_{\Hdot} \right), \nonumber \\
    \langle \hat f, \hat N_3 \rangle_{L^2_\nu} &\leq \frac{1}{2} C_5 \int_{\R} d\xi \; e^{2\nu t\abs{\xi}} e^{-h_2\abs{\xi}} \abs{\hat f(\xi)} \abs{N(\hat \Omega_2)(\xi)} \nonumber \\
    &\leq \frac{1}{2} C_5 \abs{A_\kappa} \lambda_4 \norm{\Omega_1}_{\Hdot} \norm{f}_{\Hdot}, \nonumber \\
    \langle \hat f, \hat N_4 \rangle_{L^2_\nu} &\leq \pi \sum_{n=1}^\infty \dnorm{e^{-h_2\abs{\xi}} \frac{|\xi|^{n}}{n!}}_{L^\infty} (n+1) \norm{f}_{\fzeronenu}^n \norm{\Omega_2}_{\Hdot} \norm{f}_{\Hdot}, \nonumber
\end{align}
in which 
\begin{align}
    \lambda_5, \lambda_7 \to 0 \quad \text{and} \quad \lambda_6 \lesssim 1 \quad \text{as} \quad \norm{f}_{\fzeronenu} + \norm{f}_{\foneonenu} \to 0.
\end{align}

Related to the the computations in \eqref{prodfN2},
\begin{align*}
     \langle \hat f, \hat N_0 \rangle_{L^2_\nu} &\leq \frac{1}{2} \norm{\Omega_1}_{\Hdot}\sum_{n=0}^\infty \norm{f}_{\foneonenu}^{2n+2} \norm{f}_{\Hdot} + (2n+2) \norm{f}_{\foneonenu}^{2n+1} \norm{f}_{\fthreehalvesonenu}\norm{f}_{L^2_\nu},\\
    \langle \hat f, N(\widehat{\partial_\alpha \Omega_1 })\rangle_{L^2_\nu} &\leq \norm{\Omega_1}_{\Hdot} (\lambda_5 \norm{f}_{\Hdot} + \lambda_6 \norm{f}_{\fthreehalvesonenu} \norm{f}_{L^2_\nu} ) + \lambda_7  \norm{\Omega_2}_{\Hdot} \norm{f}_{\Hdot}.
\end{align*}
So it is enough to control $\Omega_1$ and $\Omega_2$. We can write
\begin{align*}
    \hat \Omega_1 = \hat \Omega_{11} + \hat \Omega_{12} - 2A_\rho \hat f
\end{align*}
Using \eqref{omega12fourierlinearpartsum} and \eqref{Omega11abs} and  we have
\begin{align}
    \norm{\Omega_{11}}_{\Hdot}  &\leq \abs{A_\mu} \sum_{n=0}^\infty \norm{f}^{2n+1}_{\foneonenu} \norm{\Omega_1}_{\Hdot} + (2n + 1) \norm{f}_{\foneonenu}^{2n} \norm{f}_{\fthreehalvesonenu}\norm{\Omega_1}_{L^2_\nu} \nonumber\\
    \label{Omega11Hdotnu}
    &\defeq \abs{A_\mu} (C_6\norm{\Omega_1}_{\Hdot} + C_7 \norm{\Omega_1}_{L^2_\nu})\\
    \label{Omega12Hdotnu}
    \norm{\Omega_{12}}_{\Hdot} &\leq \abs{A_\mu} \Big(C_0 + C_8 \Big) \norm{\Omega_2}_{\Hdot}
\end{align}
where 
\begin{align*}
    C_7 &\to \norm{f}_{\fthreehalvesonenu},\\
    C_8 &= \sum_{n = 1}^\infty n \bnorm{e^{-h_2\abs{\xi}} \frac{\abs{\xi}^{n-1/2}}{n!} }_{L^\infty} \norm{f}_{\fzeronenu}^{n-1}\norm{f}_{\fonehalfonenu} \to \frac{1}{\sqrt{2 e h_2} } \norm{f}_{\fonehalfonenu} \to 0
\end{align*}
as $\norm{f}_{\fzeronenu} + \norm{f}_{\foneonenu} \to 0$. 
With \eqref{Omega21fourierexpansion} and \eqref{Omega22fourierexpansion} we have the bounds
\begin{align}
    \norm{\Omega_2}_{\Hdot} &\leq \abs{A_\kappa} C_0(1 + \norm{f}_{\foneonenu})\norm{\Omega_1}_{\Hdot} \nonumber \\
    &\hspace{2.5em} + \abs{A_\kappa} \Big(\frac{C_2}{\norm{f}_{\fzeronenu}} \norm{f}_{\fonehalfonenu} (1 + \norm{f}_{\foneonenu}) + C_0 \norm{f}_{\fthreehalvesonenu}\Big) \norm{\Omega_1}_{L^2_\nu} \nonumber \\
    \label{Omega2Hdotnu}
    &\eqdef \abs{A_\kappa}( C_9 \norm{\Omega_1}_{\Hdot} + C_{10}\norm{\Omega_1}_{L^2_\nu}) 
\end{align}
with $C_9 \to 1$ and $C_{10} \to \abs{A_\kappa} \norm{f}_{\fthreehalvesonenu}$ as $\norm{f}_{\fzeronenu} + \norm{f}_{\foneonenu} \to 0$.
Similarly to the above calculations, we derive
\begin{align*}
    \norm{\Omega_2}_{L^2_\nu} &\leq \abs{A_\kappa} C_0(1 + \norm{f}_{\foneonenu}) \norm{\Omega_1}_{L^2_\nu} = \abs{A_\kappa} C_9 \norm{\Omega_1}_{L^2_\nu}
\end{align*}
and therefore by \eqref{omega12fourierlinearpartsum} and \eqref{Omega11abs} we have
\begin{align*}
    (1-\abs{A_\mu} C_6)\norm{\Omega_1}_{L^2_\nu} &\leq  \abs{A_\mu} C_0 \norm{\Omega_2}_{L^2_\nu} + 2A_\rho \norm{f}_{L^2_\nu}\\
    &\leq |A_\kappa| \abs{A_\mu} C_0 C_9 \norm{\Omega_1}_{L^2_\nu} + 2A_\rho \norm{f}_{L^2_\nu} 
\end{align*}
to give us
\begin{align}
    \label{Omega1elltwonu}
    \norm{\Omega_1}_{L^2_\nu} &\leq 2A_\rho C_{11} \norm{f}_{L^2_\nu},\\
    \label{Omega2elltwonu}
    \norm{\Omega_2}_{L^2_\nu} &\leq 2A_\rho \abs{A_\kappa} C_9 C_{11} \norm{f}_{L^2_\nu}
\end{align}
where
\begin{equation}\label{C11}
C_{11} = (1-\abs{A_{\mu}}C_{6}-\abs{A_{\kappa}}\abs{A_{\mu}}C_{0}C_{9})^{-1} \to (1- \abs{A_\kappa} \abs{A_\mu} )\inv \quad \text{as} \quad \norm{f}_{\fzeronenu} + \norm{f}_{\foneonenu} \to 0.
\end{equation}
Collecting terms from \eqref{Omega11Hdotnu}, \eqref{Omega12Hdotnu}, \eqref{Omega2Hdotnu}, \eqref{Omega1elltwonu}, and \eqref{Omega2elltwonu} we find
\begin{align}
    \norm{\Omega_1}_{\Hdot} &\leq 2A_\rho C_{12} ( |A_\mu| C_{11} (C_7 + \abs{A_\kappa} (C_0 + C_8) C_{10}) \norm{f}_{L^2_\nu} + \norm{f}_{\Hdot})\nonumber\\
    \label{Omega1Hdot}
    &\eqdef 2A_\rho C_{12} (C_{13} \norm{f}_{L^2_\nu} + \norm{f}_{\Hdot})\\
    \norm{\Omega_2}_{\Hdot} &\leq 2A_\rho \abs{A_\kappa} \Big((C_9C_{12}C_{13} + C_{10}C_{11}) \norm{f}_{L^2_\nu} + C_9C_{12} \norm{f}_{\Hdot} \Big) \nonumber\\
    \label{Omega2Hdot}
    &\eqdef 2A_\rho \abs{A_\kappa} (C_{14} \norm{f}_{L^2_\nu} + C_9C_{12} \norm{f}_{\Hdot})
\end{align}
for
\begin{align*}
    C_{12} &= \Big( 1- \abs{A_\mu} (C_6 + \abs{A_\kappa} (C_0 + C_8) C_9 )\Big)\inv \to (1-\abs{A_\kappa} \abs{A_\mu} )\inv ,\\
    C_{13} &\lesssim 1+\norm{f}_{\fthreehalvesonenu},\\
    C_{14} &\lesssim 1+ \norm{f}_{\fthreehalvesonenu}
\end{align*}

We see two types of expressions. For the first type of expression, of the form $\norm{\Omega_i}_{\Hdot} \norm{f}_{L^2_\nu}$, we can apply Young's inequality for products and control the $\Hdot$ terms as in \eqref{youngsexample}, e.g. using \eqref{Omega1Hdot}:
\begin{align*}
    \norm{\Omega_1}_{\Hdot} \norm{f}_{L^2_\nu} &\leq  2A_\rho C_{12} (C_{13} \norm{f}_{L^2_\nu}^{2}+ \norm{f}_{\Hdot}\norm{f}_{L^2_\nu})\\
    &\leq \frac{C}{\varepsilon_{n}}\norm{f}_{L^2_\nu}^{2} + \varepsilon_{n}\norm{f}_{\Hdot}^{2}.
\end{align*}
Here $\varepsilon_{n}$ as earlier in \eqref{youngsexample} can always be chosen arbitrarily small. For the second type of expression $\norm{\Omega_i}_{\Hdot} \norm{f}_{\Hdot}$, after applying \eqref{Omega1Hdot} or \eqref{Omega2Hdot} we can control the resulting $\norm{f}_{\Hdot}^2$ term via the linear decay term in the interface equation, see $\sigma_{2}$ below. In terms with $\norm{f}_{L^{2}_{\nu}}$ that contain coefficient of $\norm{f}_{\fthreehalvesonenu}^2$, such as the middle term in \eqref{youngsexample}, we use $\norm{f}_{\fthreehalvesonenu}^2 \leq \norm{f}_{\foneonenu}\norm{f}_{\ftwonenu}$ by \eqref{interpolation}. Collecting terms from above, \eqref{interfaceL2} becomes
\begin{align*}
    \frac{1}{2} \frac{d}{dt}\norm{f}_{L^2_\nu}^2(t) &\leq  (-A_\rho \theta + \nu + A_\rho \sigma_2 + \varepsilon) \norm{f}_{\Hdot}^2  + (R_{1}+ R_{2}\cdot(\|f\|_{\fthreehalvesonenu}+\|f\|_{\ftwonenu})) \norm{f}_{L^2_\nu}^{2}
\end{align*}
where
\begin{align}
    \sigma_2 &\eqdef \sigma_2 (\norm{f}_{\fzeronenu}, \norm{f}_{\foneonenu}) \nonumber\\
    \begin{split}
    \label{sigma2}
    &=  C_{12}(C_6 + C_0 + C_2) \norm{f}_{\foneonenu}\\
    &\hspace{2.5em} + |1+A_\mu| C_{12} \left( 2\abs{A_\kappa}\lambda_4  +  C_5 (\lambda_5  + \abs{A_\kappa}  ( \lambda_7C_9 + \lambda_4 )) \right) \\
    &\hspace{3.5em} + 2\pi\abs{A_\kappa}(C_0-1 + C_2)  C_9C_{12} + C_6C_{12} \norm{f}_{\foneonenu} + 2\lambda_5 C_{12} + 2\abs{A_\kappa}\lambda_7 C_9C_{12},
    \end{split}
\end{align}
in which $\varepsilon = \varepsilon(\norm{f_{0}}_{\fzerone}, \norm{f_{0}}_{\foneone})$ is arbitrarily small and $R_{i}=  R_{i}(\norm{f_0}_{\fzeronenu}, \norm{f_0}_{\foneonenu})$ for $i=1,2$ is bounded for medium sized initial data. Again, recall that $\norm{f}_{\fthreehalvesonenu}^2 \leq \norm{f}_{\foneonenu}\norm{f}_{\ftwonenu}$ by \eqref{interpolation}. Hence, by \eqref{maininequality}, $\|f\|_{\foneonenu}(t)$ and $\|f\|_{\ftwonenu}(t)$ are $L^{1}$ functions in time on $[0,T]$ for any $T>0$, and hence, so is $\|f\|_{\fthreehalvesonenu}(t)$ with $L^{1}$ norm bounded by initial data. By Gronwall's inequality, we obtain \eqref{L2inequality}.

\subsection{Sobolev Space Estimates}
In this section, we will prove an evolution estimate for a subcritical Sobolev norm $H^{\frac{3}{2}+\epsilon}$ of the interface 
\begin{align*}
        \frac{1}{2} \frac{d}{dt}\norm{f}_{\dot H^{\frac{3}{2} + \epsilon}}^2(t) =- A_\rho \int_\R  \abs{\xi}^{4+2\epsilon} \abs{\hat f(\xi)}^2 \left(1- \frac{A_\kappa(1-A_\mu)}{e^{2h_2\abs{\xi}} - A_\kappa A_\mu}\right) \; d\xi + \langle \hat f, \text{h.o.t.} \rangle_{H^{\frac{3}{2} + \epsilon}}.
\end{align*}
The weighted triangle inequality
\begin{align*}
\abs{\sum_{j=1}^{n} a_{j}}^{1+\epsilon} \leq n^{\epsilon}\sum_{j=1}^{n} \abs{a_{j}}^{1+\epsilon}
\end{align*}
allows us to use the constants from prior sections with the convention that $C_{k,\epsilon}$ is $C_k$ adjusted for the small $\epsilon$ weight from the weighted triangle inequality. All of the $C_{k,\epsilon}$ are bounded and $C_{k,\epsilon} \rightarrow C_{k}$ as $\epsilon \rightarrow 0$.
By \eqref{omega2fourierexpansion} we have
\begin{align*}
        \norm{\Omega_2}_{\dot H^{2+\epsilon}} &= \norm{\omega_2}_{\dot H^{1 + \epsilon}} \\
        &\leq \abs{A_\kappa } \Big( \frac{C_{2,\epsilon}}{\norm{f}_{\fzerone}} \norm{\omega_1}_{\fzerone} \norm{f}_{\dot H^{1+\epsilon}} + C_{0,\epsilon} \norm{\Omega_1}_{\dot H^{2+\epsilon}} \Big).
\end{align*}
where 
\begin{equation*}C_{0,\epsilon}= \sum_{n = 0}^\infty \frac{n^{n}(n+1)^{\epsilon}}{e^nn!} \left(\frac{\norm{f}_{\fzeronenu}}{h_2}\right)^n, \qquad C_{2,\epsilon} = \sum_{n = 1}^\infty  \frac{n^{n+1}(n+1)^{\epsilon}}{e^nn!} \left(\frac{\norm{f}_{\fzeronenu}}{h_2}\right)^n
\end{equation*}
Similarly, by \eqref{omega11est} and \eqref{omega12calc}
\begin{align*}
\norm{\Omega_1}_{\dot H^{2+\epsilon}} &\leq \abs{A_\mu} \Big( 2\frac{1 + \norm{f}_{\foneone}^2}{(1-\norm{f}_{\foneone}^2)^2} \norm{\omega_1}_{\fzerone} \norm{f}_{\dot H^{2+\epsilon}} + 2C_{6,\epsilon} \norm{\Omega_1}_{\dot H^{2+\epsilon}} + C_{0,\epsilon} (1+\norm{f}_{\foneone}) \norm{\Omega_2}_{\dot H^{2+\epsilon}}\\
& \hspace{5em} + \frac{C_{2,\epsilon}}{\norm{f}_{\fzerone}}(1+\norm{f}_{\foneone}) \norm{\omega_2}_{\fzerone} \norm{f}_{\dot H^{1+\epsilon}} + C_{0,\epsilon} \norm{\omega_2}_{\fzerone} \norm{f}_{\dot H^{2+\epsilon}}\Big) + 2A_\rho \norm{f}_{\dot H^{2+\epsilon}}
    \end{align*}
from these we derive
    \begin{align}
        \label{Omega1H2}
        \norm{\Omega_1}_{\dot H^{2+\epsilon}} \leq 2A_\rho C_{1,\epsilon} (C_{15,\epsilon} \norm{f}_{\dot H^{1+\epsilon}} + C_{16,\epsilon} \norm{f}_{\dot H^{2+\epsilon}})
    \end{align}
where for example, we recover the $\epsilon$ adjusted constant
    \[
        C_{1,\epsilon} = \Big(1 - 2\abs{A_\mu}C_{6,\epsilon} - \abs{A_\mu}\abs{A_\kappa} C_{0,\epsilon}^2 (1+\norm{f}_{\foneone}) \Big)\inv
    \]
and $\lim_{\epsilon\rightarrow 0}C_{16,\epsilon} < C_{3} $. This in turn implies
    \begin{align}
        \label{Omega2H2}
        \norm{\Omega_2}_{\dot H^{2+\epsilon}} \leq 2A_\rho (C_{17,\epsilon}\norm{f}_{\dot H^{1+\epsilon}} + C_{18,\epsilon} \norm{f}_{\dot H^{2+\epsilon}}).
    \end{align}
where $C_{18,\epsilon} = C_{0,\epsilon}C_{1,\epsilon} C_{16,\epsilon}$ which as $\epsilon \rightarrow 0$ approaches a value strictly less than $C_{1}C_{4}$.

Using \eqref{eq: I_2}, we have that
\begin{align*}
\langle \hat f, \hat I_2 \rangle_{\dot H^{\frac{3}{2}+\epsilon}} &\leq \pi \int_{\mathbb{R}} \abs{\xi}^{3+2\epsilon}\abs{\hat{f}(\xi)} \sum_{n=0}^{\infty}(\ast^{2n+2}\abs{\widehat{\partial_\alpha f}} \ast (\abs{\cdot}\abs{\Omega_{1}}))(\xi) d\xi\\
&\leq \pi \norm{f}_{\dot{H}^{2+\epsilon}} \sum_{n=0}^{\infty}\bnorm{\ast^{2n+2}\abs{\widehat{\partial_\alpha f}} \ast \abs{\cdot}\abs{\Omega_{1}}}_{\dot{H}^{1+\epsilon}}
\end{align*}
Applying the weighted triangle inequality and then using \eqref{Omega1H2} and \eqref{Omega2H2} we obtain
\begin{align*}
  \langle \hat f, \hat I_2 \rangle_{\dot H^{\frac{3}{2} + \epsilon}} &\leq \pi \norm{f}_{\dot H^{2+\epsilon}}\sum_{n=0}^\infty  (2n+3)^{\epsilon}\Big((2n+2) \norm{f}_{\foneone}^{2n+1} \norm{\Omega_1}_{\foneone} \norm{f}_{\dot H^{2+\epsilon}} + \norm{f}_{\foneone}^{2n+2} \norm{\Omega_1}_{\dot H^{2+\epsilon}}\Big)
\end{align*}
Similarly, making use of the fact that the $\norm{N(\Omega_i)}_{H^{2 + \epsilon}}$ can be bounded in the same way as \eqref{Omega1H2} or \eqref{Omega2H2}
\begin{align*}
     \langle \hat f, \hat I_4 \rangle_{\dot H^{\frac{3}{2} + \epsilon}} &\leq \pi \norm{f}_{\dot H^{2+\epsilon}} \Big( C_{0,\epsilon} (\norm{f}_{\foneonenu} \norm{\Omega_2}_{\dot H^{2+\epsilon}} + \norm{\omega_2}_{\fzerone} \norm{f}_{\dot H^{2+\epsilon}})\\
     &\hspace{50mm}+ \frac{C_{2,\epsilon}}{\norm{f}_{\fzerone}} \norm{f}_{\foneone} \norm{\omega_2}_{\fzerone} \norm{ f}_{\dot H^{1+\epsilon}}\Big)\\
     \langle \hat f, \hat N_0 \rangle_{\dot H^{\frac{3}{2} + \epsilon}} &\leq \frac{1}{2} \norm{f}_{\dot H^{2+\epsilon}} (2n+3)^\epsilon \Big((2n+2) \norm{f}_{\foneone}^{2n+1} \norm{f}_{\foneone} \norm{\Omega_1}_{\dot H^{1+\epsilon}} + \norm{f}_{\foneone}^{2n+2} \norm{\Omega_1}_{\dot H^{2+\epsilon}} \Big)\\
     \langle \hat f, \hat N_1 \rangle_{\dot H^{\frac{3}{2} + \epsilon}} &\leq \frac{1}{2} \norm{f}_{\dot H^{2+\epsilon}} \norm{N( \Omega_2)}_{\dot H^{2+\epsilon}}\\
     \langle \hat f, \hat N_2 \rangle_{\dot H^{\frac{3}{2} + \epsilon}} &\leq \frac{1}{2} \bnorm{\frac{A_\kappa}{e^{2h_2 \abs \xi} - A_\kappa A_\mu} }_{L^\infty} \norm{f}_{\dot H^{2+\epsilon}} \norm{N( \Omega_1)}_{\dot H^{2+\epsilon}}\\
     \langle \hat f, \hat N_3 \rangle_{\dot H^{\frac{3}{2} + \epsilon}} &\leq \frac{1}{2} C_5 \norm{f}_{\dot H^{2+\epsilon}} \norm{N( \Omega_2)}_{\dot H^{2+\epsilon}}\\
     \langle \hat f, \hat N_4 \rangle_{\dot H^{\frac{3}{2} + \epsilon}} &\leq \pi \norm{f}_{\dot H^{2+\epsilon}} \Big((C_{0,\epsilon}-1) \norm{\Omega_2}_{\dot H^{2+\epsilon}} + \frac{C_{2,\epsilon}}{\norm{f}_{\fzerone}} \norm{f}_{\foneone} \norm{\Omega_2}_{\dot H^{1+\epsilon}} \Big)\\
    \langle \hat f, N(\widehat{\partial_\alpha \Omega_1}) \rangle_{\dot H^{\frac{3}{2} + \epsilon}} &\leq \norm{f}_{\dot H^{2+\epsilon}} \norm{N( \Omega_1)}_{\dot H^{2+\epsilon}}
\end{align*}
    
    Combining terms via Young's inequality for products we have
    \begin{align}
    \label{sobolevestimate}
    \frac{1}{2} \frac{d}{dt}\norm{f}_{\dot H^{\frac{3}{2} +\epsilon} }^2(t) &\leq  -A_\rho( \theta - \tilde \sigma_{\epsilon}) \norm{f}_{\dot H^{2+\epsilon}}^2  + S(\norm{f}_{\fzerone}, \norm{f}_{\foneone}) \norm{f}_{\dot H^{1+\epsilon}}^2
    \end{align}
    In which, as $\epsilon \to 0$
\begin{align}
    \tilde \sigma_{\epsilon} < \sigma_1
\end{align}
and $S$ is a rational function of $\norm{f}_{\fzerone}$ and $\norm{f}_{\foneone}$ that vanishes at 0. So for every finite $T > 0$ we can bound $\frac{d}{dt}\norm{f}_{\dot H^{\frac{3}{2}+\epsilon}}(t)$ on $[0,T]$ giving the desired derivative in time bound.

\section{Proof of Theorem}
\label{sec:existence}
We argue similarly to \cite{GGPS} for the proof of existence and uniqueness of solutions to \eqref{eq: interface}. Uniqueness of solutions is proven at the level of $\fzerone$ and follows exactly as in \cite{GGPS}. It yields an inequality of the type
\begin{align}\label{uniqueness}
\frac{d}{dt}\norm{f-g}_{\fzerone} &\leq -c_{1}\norm{f-g}_{\foneone} + c_{2}\norm{f-g}_{\fzerone}
\end{align}
 for two solutions $f,g$ of \eqref{eq: interface} and for some constants $c_{i}= c_{i}(\norm{f_{0}}_{\fzerone},\norm{f_{0}}_{\foneone},\norm{g_{0}}_{\fzerone},\norm{g_{0}}_{\foneone})>0$. Note that the key difference with \cite{GGPS} is that for some terms, the difference of solutions occurs at the level of $\fzerone$ which, unlike the Muskat problem without a permeability jump, can not be absorbed into the decay term.
 
 Consider the mollified system with initial data $f_{0}^{\varepsilon} = \varphi^{\varepsilon} \ast f_{0}$ and the evolution equation
 \begin{equation}\label{existenceequation}
      \partial_t f^{\epsilon} =  \mathcal{L}(\varphi^{\varepsilon}\ast\varphi^{\varepsilon}\ast f^{\epsilon}) + \varphi^{\varepsilon}\ast\mathcal{N}(\varphi^{\varepsilon}\ast\varphi^{\varepsilon}\ast f^{\epsilon}, \omega_{1}^{\varepsilon},\omega_{2}^{\varepsilon})
 \end{equation}
 where
 \begin{align*}
    \widehat{\mathcal{L}g}(\xi) = - A_\rho |\xi| \hat g(\xi) \left(1 - \frac{A_\kappa(1-A_\mu) }{e^{2h_2\abs{\xi}} - A_\kappa A_\mu}\right)
 \end{align*}
 and $\mathcal{N}$ is the remaining nonlinear terms from \eqref{interface_neat}. Here $\omega_{i}^{\varepsilon} = \partial_{\alpha}\Omega_i^{\varepsilon}$ are given by the mollified \eqref{omega1equation}, \eqref{omega2equation}, \eqref{Omega1}, and \eqref{Omega2} where $f$ is replaced by the mollified $\varphi^{\varepsilon}\ast\varphi^{\varepsilon}\ast f^{\varepsilon}$. Given initial data of medium size from Theorem \ref{MainTheorem} and because $H^{\frac{3}{2} + \epsilon} \hookrightarrow \fzerone\cap\foneone$, the mollified system satisfies the hypothesis to apply Picard's theorem. We get a local solution $f^{\varepsilon} \in C([0,T_{\varepsilon}); H^{\frac{3}{2} + \epsilon})$. Next, we can reproduce the analogous estimates for the medium size condition on $f_{0}$ for $0\leq s \leq 1$:
\begin{equation}\label{existenceinequality}
\norm{f^{\varepsilon}}_{\mathcal{F}^{s,1}_{\nu}}(t) +(A_\rho\theta  - A_{\rho} \sigma_{s} -\nu)\int_{0}^{t} \norm{\varphi^{\varepsilon}\ast\varphi^{\varepsilon}\ast f^{\varepsilon}}_{\mathcal{F}^{s+1,1}_{\nu}}(s)\; ds \leq \|f_{0}\|_{\mathcal{F}^{s,1}_{0}}
\end{equation}
and
\begin{equation}\label{eq:epsilonL2}
  \norm{f^{\varepsilon}}_{L^2_\nu}^2 (t) \leq \norm{f_0}_{L^2}^2 \cdot \exp(R(\norm{f_0}_{\fzerone}, \norm{f_0}_{\foneone}))
\end{equation}

Due to the exponential weight in $L^{2}_{\nu}$, the estimate \eqref{eq:epsilonL2} implies that $\|f^{\varepsilon}\|_{H^{\frac{3}{2}+\epsilon}}(t) \leq  C_{\epsilon}(t)\norm{f_0}_{L^2}^2 \cdot \exp(R(\norm{f_0}_{\fzerone}, \norm{f_0}_{\foneone}))$ for $t>0$ where $C_{\epsilon}(t)$ is a bounded decreasing constant in $t>0$. Moreover, it can be seen combining the proof of \eqref{L2inequality} and \eqref{sobolevestimate} that
$$\frac{d}{dt}\|f^{\varepsilon}\|_{H^{\frac{3}{2}+\epsilon}}^{2} \leq G(\|f^{\varepsilon}\|_{H^{\frac{3}{2}+\epsilon}}, \|f_{0}\|_{\fzerone}, \|f_{0}\|_{\foneone})$$
for a continuous function $G$. Hence, the local solution can be extended to $C([0,T]; H^{\frac{3}{2}+\epsilon})$ for any $T>0$.

By \eqref{uniqueness} and following the argument in \cite{GGPS}, the sequence $f^{\varepsilon_{n}}$ is shown to be Cauchy in $L^{\infty}([0,T];\fzerone)$ for any $\varepsilon_{n} \rightarrow 0$. The main idea is that by the argument from uniqueness, we have
\begin{align*}
\frac{d}{dt}\norm{f^{\varepsilon} - f^{\varepsilon '}}_{\fzerone} &\leq -c_{1}\norm{ \varphi^{\varepsilon} \ast \varphi^{\varepsilon} \ast f^{\varepsilon} -  \varphi^{\varepsilon'} \ast \varphi^{\varepsilon'} \ast f^{\varepsilon '}}_{\foneone} + c_{2}\norm{ \varphi^{\varepsilon} \ast \varphi^{\varepsilon} \ast f^{\varepsilon} -  \varphi^{\varepsilon'} \ast \varphi^{\varepsilon'} \ast f^{\varepsilon '}}_{\fzerone}
\end{align*}
and hence
\begin{align*}
\norm{f^{\varepsilon} - f^{\varepsilon '}}_{\fzerone}(t) \leq \norm{\varphi^{\varepsilon} \ast f_{0} - \varphi^{\varepsilon'} \ast f_{0}}_{\fzerone} + c_{2} \int_{0}^{t}\norm{ \varphi^{\varepsilon} \ast \varphi^{\varepsilon} \ast f^{\varepsilon} -  \varphi^{\varepsilon'} \ast \varphi^{\varepsilon'} \ast f^{\varepsilon '}}_{\fzerone}(s) ds.
\end{align*}
Using the Mean Value Theorem in the mollifiers in the Fourier variables and assuming $\varepsilon \geq \varepsilon'$
\begin{align*}
    \norm{\varphi^{\varepsilon} \ast f_{0} - \varphi^{\varepsilon'} \ast f_{0}}_{\fzerone} \leq C\norm{f_{0}}_{\foneone}\varepsilon^{\frac{1}{2}}
\end{align*}
and 
\begin{align*}
   \norm{ \varphi^{\varepsilon} \ast \varphi^{\varepsilon} \ast f^{\varepsilon} -  \varphi^{\varepsilon'} \ast \varphi^{\varepsilon'} \ast f^{\varepsilon '}}_{\fzerone} &\leq  \norm{ \varphi^{\varepsilon} \ast \varphi^{\varepsilon} \ast f^{\varepsilon} -  \varphi^{\varepsilon} \ast \varphi^{\varepsilon} \ast f^{\varepsilon '}}_{\fzerone} +  \norm{ \varphi^{\varepsilon} \ast \varphi^{\varepsilon} \ast f^{\varepsilon'} -  \varphi^{\varepsilon'} \ast \varphi^{\varepsilon'} \ast f^{\varepsilon '}}_{\fzerone}\\
   &\leq \norm{ f^{\varepsilon} -  f^{\varepsilon '}}_{\fzerone} +  C\norm{ f^{\varepsilon '}}_{\foneone}\varepsilon^{\frac{1}{2}}\\
     &\leq \norm{ f^{\varepsilon} -  f^{\varepsilon '}}_{\fzerone} + C\norm{f_{0}}_{\foneone}\varepsilon^{\frac{1}{2}}.
\end{align*}
Thus,
\begin{align}\label{cauchy}
\norm{f^{\varepsilon} - f^{\varepsilon '}}_{\fzerone}(t) \leq C(1+c_{2}t)\norm{f_{0}}_{\foneone}\varepsilon^{\frac{1}{2}} + c_{2} \int_{0}^{t}\norm{ f^{\varepsilon} -  f^{\varepsilon '}}_{\fzerone}(s)  ds.
\end{align}
Gronwall's inequality finally yields
\begin{align*}
\norm{f^{\varepsilon} - f^{\varepsilon '}}_{\fzerone}(t) \leq C(1+c_{2}t)e^{c_{2}t}\norm{f_{0}}_{\foneone}\varepsilon^{\frac{1}{2}}.
\end{align*}
Hence, there exists a limit $\ f^{\varepsilon}\rightarrow f$ in $L^{\infty}([0,T];\fzerone)$. Hence, we can obtain pointwise almost everywhere convergence of a subsequence $\hat f^{\varepsilon_{n}}(\xi,t)$ and $\hat{\varphi}^{\varepsilon_{n}}(\xi,t)^{2}\hat f^{\varepsilon_{n}}(\xi,t)$ to $\hat{f}(\xi,t)$. Thus, Fatou's lemma applied to \eqref{existenceinequality} allows us to conclude that the limit $f$ indeed satisfies the inequality \eqref{maininequality} for $s=0$ and $s=1$.

Interpolation \eqref{interpolation} with \eqref{cauchy} and \eqref{existenceinequality} yields strong convergence of $\varphi^{\varepsilon} \ast \varphi^{\varepsilon} \ast f^{\varepsilon}$ to $f$ in $L^{2}([0,T];\foneone)$. 
Finally, we can now take limits in \eqref{existenceequation}, we get the limiting function $f$ as the unique solution to \eqref{eq: interface}.

\subsubsection*{Acknowledgments}
The authors would like to thank the University of Michigan Research Experience for Undergraduates (REU) program for their support of Nikhil Shankar during the early stages of this project. NP was partially supported by AMS-Simons Travel Grants, which are administered by the American Mathematical Society with support from the Simons Foundation.

\printbibliography

	\vspace{0.1in}
	\begin{tabular}{ll}
		\textbf{Neel Patel} \\
		{\small Instituto de Ciencias Matem\'aticas, ICMAT} \\ 
		{\small Madrid, Spain} \\ 
		{\small Email: neeljp@umich.edu}\\
		\small{\textbf{Former Affiliation:} }\\
		{\small Department of Mathematics}\\
		{\small University of Michigan} \\ 
		{\small Ann Arbor, Michigan, 48105}
	\end{tabular}\qquad
	\begin{tabular}{ll}
		\textbf{Nikhil Shankar} \\
		{\small Department of Mathematics}\\
		{\small Duke University} \\ 
		{\small Physics Building 274E}\\
		{\small 120 Science Dr, Durham, NC 27710}\\ 
		{\small Email: nikhil.shankar@duke.edu}\\[2.45em]
	\end{tabular}

\end{document}